\documentclass[11pt]{article}
\usepackage{mathtools}
\usepackage{amsmath,amsthm}
\usepackage{csquotes}
\usepackage{amssymb}
\usepackage{graphicx}
\usepackage{subfig}
\usepackage{mathpazo}
\usepackage{soul}  
\usepackage{xcolor}
\usepackage{empheq}
\usepackage{lipsum}

\definecolor{myblue}{rgb}{.8, .8, 1}

\usepackage[obeyspaces,hyphens,spaces]{url}
\usepackage[unicode]{hyperref}
\hypersetup{
    colorlinks=true,
    linkcolor=blue,
    citecolor=blue,
    filecolor=magenta,
    urlcolor=cyan
}

\usepackage[
  open,
  openlevel=2,
  atend,
  numbered
]{bookmark}

\usepackage[capitalize, nameinlink, noabbrev]{cleveref}
\crefname{equation}{}{}
\crefname{chapter}{Chapter}{Chapters}
\crefname{item}{item}{items}
\crefname{figure}{Figure}{Figures}
\crefname{theorem}{Theorem}{Theorems}
\crefname{lemma}{Lemma}{Lemmas}
\crefname{proposition}{Proposition}{Propositions}
\crefname{corollary}{Corollary}{Corollarys}
\crefname{definition}{Definition}{Definitions}
\crefname{fact}{Fact}{Facts}
\crefname{example}{Example}{Examples}
\crefname{algorithm}{Algorithm}{Algorithms}
\crefname{remark}{Remark}{Remarks}
\crefname{note}{Note}{Notes}
\crefname{notation}{Notation}{Notations}
\crefname{case}{Case}{Cases}
\crefname{exercise}{Exercise}{Exercises}
\crefname{question}{Question}{Questions}
\crefname{claim}{Claim}{Claims}
\crefname{enumi}{}{}

\usepackage[top= 2cm, bottom = 2 cm, left = 2.2 cm, right= 2.2 cm]{geometry}  
\numberwithin{equation}{section}

\theoremstyle{plain}
\newtheorem{theorem}{Theorem}[section]
\newtheorem{corollary}[theorem]{Corollary}
\newtheorem{fact}[theorem]{Fact}
\newtheorem{lemma}[theorem]{Lemma}
\newtheorem{proposition}[theorem]{Proposition}

\theoremstyle{definition}
\newtheorem{definition}[theorem]{Definition}

\newtheorem{remark}[theorem]{Remark}

\usepackage{enumitem}

\newcommand{\zer}{\ensuremath{\operatorname{zer}}}

\newcommand{\weakly}{\ensuremath{{\;\operatorname{\rightharpoonup}\;}}}

\newcommand{\dom}{\ensuremath{\operatorname{dom}}}
\newcommand{\gra}{\ensuremath{\operatorname{gra}}}
\newcommand{\Fix}{\ensuremath{\operatorname{Fix}}}
\newcommand{\Id}{\ensuremath{\operatorname{Id}}}

\newcommand{\dist}{\ensuremath{\operatorname{d}}}
\newcommand{\Pro}{\ensuremath{\operatorname{P}}}
\newcommand{\R}{\ensuremath{\operatorname{R}}}

\newcommand{\J}{\ensuremath{\operatorname{J}}}

\newcommand{\Range}{\ensuremath{\operatorname{ran}}}

\providecommand{\abs}[1]{\left|#1\right|}
\providecommand{\norm}[1]{\left\lVert#1\right\rVert}
\providecommand{\innp}[1]{\left\langle#1\right\rangle}
\providecommand{\lr}[1]{\left(#1\right)}

\begin{document}

\title{On the Stability of Krasnosel'ski\v{\i}-Mann Iterations}

\author{
         Hui\ Ouyang\thanks{
                 Mathematics, University of British Columbia, Kelowna, B.C.\ V1V~1V7, Canada.
                 E-mail: \href{mailto:hui.ouyang@alumni.ubc.ca}{\texttt{hui.ouyang@alumni.ubc.ca}}.}
                 }

\date{March 25, 2022}

\maketitle

\begin{abstract}
	\noindent
Firstly, we invoke the weak convergence (resp.\,strong convergence) of  translated basic methods involving nonexpansive operators to establish    the weak convergence (resp.\,strong convergence) of the associated method with both perturbation and approximation. Then we employ the technique obtaining the result above  to extend convergence results from the classic Krasnosel'ski\v{\i}-Mann iterations to their relaxation variants for finding a common fixed point of associated nonexpansive operators. At last, we show applications on generalized proximal point algorithms for solving monotone inclusion problems. 
\end{abstract}

{\small
\noindent
{\bfseries 2020 Mathematics Subject Classification:}
{
	Primary 65J15,  47J25, 47H05;  
	Secondary 47H09, 47H10, 90C25. 
}

\noindent{\bfseries Keywords:}
Krasnosel'ski\v{\i}-Mann Iterations, stability, generalized proximal point algorithms,
maximally monotone operators,  firmly nonexpansiveness, weak convergence, strong convergence
}

%%%%%%%%%%%%%%%%%%%%%%%%%%%%%%%%%%%%%%%%%%%%%%%%%%%%
%%%%%%%%%%%%%%%%\section{Introduction}%%%%%%%%%%%%%%%%%%%%%
%%%%%%%%%%%%%%%%%%%%%%%%%%%%%%%%%%%%%%%%%%%%%%%%%%%%
\section{Introduction} \label{sec:Introduction}
 Throughout this paper,  
 \begin{align*}
 \text{$\mathcal{H}$ is a real Hilbert space},
 \end{align*}
 with inner product $\innp{\cdot,\cdot}$ and induced norm $\norm{\cdot}$. We  denote  the set of all nonnegative integers by $\mathbb{N} :=\{0,1,2,\ldots\}$.

 Let $T: \mathcal{H} \to \mathcal{H}$ be a nonexpansive operator and let $x_{0} \in \mathcal{H}$.  
A generalization of the \emph{Banach-Picard iteration}, $(\forall k \in \mathbb{N})$ $x_{k+1} =Tx_{k}$, is the \emph{Krasnosel'ski\v{\i}-Mann iterations}:   
\begin{align*} 
(\forall k \in \mathbb{N}) \quad x_{k+1} = x_{k} +\lambda_{k} \lr{Tx_{k} -x_{k}},
\end{align*}
where $(\forall k \in \mathbb{N})$ $\lambda_{k} \in \left[0,1\right]$. 
Let $(\forall k \in \mathbb{N})$ $T_{k} :\mathcal{H} \to \mathcal{H}$ be nonexpansive operators with $  \cap_{k \in \mathbb{N}} \Fix T_{k}  \neq \varnothing$, let $(\eta_{k})_{k \in \mathbb{N}}$ be in $\mathbb{R}_{+}$, and let $(e_{k})_{k \in \mathbb{N}}$ be a sequence of \emph{error terms} in $\mathcal{H}$. In this work, we mainly investigate the following relaxation variants of the Krasnosel'ski\v{\i}-Mann iterations. 
\begin{enumerate}
	\item The \emph{inexact  Krasnosel'ski\v{\i}-Mann iterations} 
	follows the iteration scheme:
	\begin{align*} 
	(\forall k \in \mathbb{N}) \quad x_{k+1} = x_{k} +\lambda_{k} \lr{T x_{k} -x_{k}} + \eta_{k}e_{k}.
	\end{align*}
	\item The \emph{$($exact$)$   non-stationary Krasnosel'ski\v{\i}-Mann iterations} 
	is the iteration scheme:
	\begin{align*}
	(\forall k \in \mathbb{N}) \quad x_{k+1} = x_{k} +\lambda_{k} \lr{T_{k}x_{k} -x_{k}}.
	\end{align*}
	\item 
	The \emph{inexact    non-stationary Krasnosel'ski\v{\i}-Mann iterations}   is the iteration scheme:
	\begin{align} \label{eq:KMxk+1inexact}
	(\forall k \in \mathbb{N}) \quad x_{k+1} = x_{k} +\lambda_{k} \lr{T_{k}x_{k} -x_{k}} + \eta_{k}e_{k}.
	\end{align}
	In particular, if   $(\forall k \in \mathbb{N})$ $\eta_{k}=\lambda_{k}$, then the iteration scheme  \cref{eq:KMxk+1inexact} is  always written as 
	\begin{align*} 
	(\forall k \in \mathbb{N}) \quad x_{k+1} = x_{k} +\lambda_{k} \lr{T_{k}x_{k} +e_{k} -x_{k}}.
	\end{align*}
\end{enumerate}

Notice that  Krasnosel'ski\v{\i}-Mann iterations  can be specified as the generalized  proximal point algorithm \cite{EcksteinBertsekas1992}, 
the forward-backward splitting algorithm \cite{Mercier1979}, the  Douglas-Rachford splitting algorithm \cite{DouglasRachford1956,LionsMercier1979},  the Peaceman-Rachford splitting algorithm \cite{PeacemanRachford1955},
the alternating direction method of multipliers \cite{GabayMercier1976}, the three-operator splitting schemes \cite{DavisYin2017} and so on
(see, e.g., \cite{LiangFadiliPeyre2016,Matsushita2017} for details). 
By studying relaxation variants of the Krasnosel'ski\v{\i}-Mann iterations, we are able to produce convergence results on generalized versions of optimization algorithms mentioned above if we specify the associated nonexpansive operators accordingly. 

\emph{The goal of this work is to investigate the stability of Krasnosel'ski\v{\i}-Mann iterations. In particular, we shall deduce the weak and strong convergence of relaxation variants of the Krasnosel'ski\v{\i}-Mann iterations from the corresponding convergence results on the classic Krasnosel'ski\v{\i}-Mann iterations for finding a common fixed point  of associated nonexpansive operators. Moreover, we will also apply our convergence results to generalized proximal point algorithms for finding a zero of the related maximally monotone operator.}

We summarize main results of this work as follows. 
\begin{itemize}
	\item[\textbf{R1:}] Applying \cref{theorem:GkWeakStrongConvergence}, we can deduce  the weak convergence (resp.\,strong convergence) of the iteration method (involving nonexpansive operators) with both approximation and perturbation from 
		 the weak convergence (resp.\,strong convergence) of associated translated basic methods.  
	
	\item[\textbf{R2:}] One $R$-linear convergence result on the exact version of  non-stationary Krasnosel'ski\v{\i}-Mann iterations is shown in \cref{theorem:KMI}. In addition, the weak and strong convergence of the inexact   non-stationary Krasnosel'ski\v{\i}-Mann iterations  are given in \Cref{theorem:TkxikConver,proposition:KMI}. 
	
	\item[\textbf{R3:}] \cref{theorem:exactGPPA} establishes a $R$-linear convergence result on the exact version of generalized proximal point algorithms. 
Moreover, \Cref{theorem:JckAweakconverckc,theorem:JckAweakconverckc,theorem:GPPAWeakStability,theorem:inexactGPPAstrong} provide sufficient conditions for the weak and strong convergence of generalized proximal point algorithms.
	
\end{itemize}

This work is organized as follows. In \cref{sec:AuxiliaryResults},  except for some fundamental definitions and facts, we also show some auxiliary results to facilitate proofs in subsequent sections. In \cref{section:PerturbedApproximate}, we specify \cite[Lemma~2.1]{Lemaire1996} (that   originates from  \cite[Remark~14]{BrezisLions1978}) to weak and norm topologies in Hilbert spaces and deduce the weak convergence (resp.\,strong convergence) of the perturbed and approximate iteration method involving nonexpansive operators   from 
the weak convergence (resp.\,strong convergence) of associated translated basic methods. 
This  technique on the proof of convergence is invoked in \cref{sec:ConvergenceKMI} where we extend some weak and strong convergence results from the classic Krasnosel'ski\v{\i}-Mann iterations to their relaxation variants.   In  \cref{section:GPPA}, we apply main results in \Cref{section:PerturbedApproximate,sec:ConvergenceKMI} to generalized proximal point algorithms for solving monotone inclusion problems. At last, we summarize this work and provide some directions on future work in
\cref{section:Conclusion}.

We now turn to the notation used in this work. $\Id$ stands for the \emph{identity mapping}. 
Denote by  $\mathbb{R}_{+}:=\{\lambda \in \mathbb{R} ~:~ \lambda \geq 0 \}$ and $\mathbb{R}_{++}:=\{\lambda \in \mathbb{R} ~:~ \lambda >0 \}$.  
Let $\bar{x} $ be in $ \mathcal{H}$ and let $r \in \mathbb{R}_{+}$.
$B[\bar{x};r]:= \{ y\in \mathcal{H} ~:~ \norm{y-\bar{x}} \leq r \}$ is the \emph{closed ball centered at $\bar{x}$ with radius $r$}.  
Let $C$ be a nonempty set of $\mathcal{H}$. Then $(\forall x \in \mathcal{H})$ $\dist \lr{x, C} = \inf_{y \in C} \norm{x-y}$. If $C$ is   nonempty closed and convex, then the \emph{projector} (or \emph{projection operator}) onto $C$ is the operator, denoted by $\Pro_{C}$,  that maps every point in $\mathcal{H}$ to its unique projection onto $C$.  $C$ is \emph{affine} if $C \neq \varnothing$ and $(\forall x \in C)$ $(\forall y \in C )$ $(\forall \lambda \in \mathbb{R})$ $\lambda x + (1-\lambda) y \in C$.
Let $\mathcal{D} $ be a nonempty subset of $\mathcal{H}$ and let $T: \mathcal{D} \to  \mathcal{H}$. $\Fix T :=\{ x \in \mathcal{D}~:~ x = T(x) \}$ is the \emph{set of fixed points  of $T$}. 
Let $A: \mathcal{H} \to 2^{\mathcal{H}}$ be a set-valued operator. Then $A$ is characterized by its \emph{graph} $\gra A:= \{ (x,u) \in \mathcal{H} \times \mathcal{H} ~:~ u\in A(x) \}$. The \emph{inverse} of $A$, denoted by $A^{-1}$, is defined through its graph $\gra A^{-1} :=  \{ (u,x) \in \mathcal{H} \times \mathcal{H} ~:~ (x,u) \in \gra A \}$.
The  \emph{domain},  \emph{range}, and \emph{set of zeros}  of $A$   are  defined by $\dom A := \left\{ x \in \mathcal{H} ~:~ Ax \neq \varnothing \right\}$,  $\Range A := \left\{ y\in \mathcal{H} ~:~ \exists~ x \in \mathcal{H} \text{ s.t. } y\in Ax \right\}$, and $\zer A :=  \left\{ x \in \mathcal{H} ~:~ 0 \in Ax \right\}$, respectively. 
Let $(y_{k})_{k \in \mathbb{N}}$ be a sequence in $\mathcal{H}$.  $\Omega \lr{ \lr{y_{k}}_{k \in \mathbb{N}} }$ stands for the \emph{set of all weak sequential clusters of the sequence $(y_{k})_{k \in \mathbb{N}}$}.  If  $(y_{k})_{k \in \mathbb{N}}$ \emph{converges strongly} to $\bar{y}$, then we denote by $y_{k} \to \bar{y}$.  $(y_{k})_{k \in \mathbb{N}}$ \emph{converges weakly} to  $\bar{y} $ if, for every $u \in \mathcal{H}$, $\innp{y_{k},u} \rightarrow \innp{y,u}$; in symbols, $y_{k} \weakly \bar{y}$. Suppose that $(y_{k})_{k \in \mathbb{N}}$ converges to a point $\bar{y} \in \mathcal{H}$. 
 $(y_{k})_{k \in \mathbb{N}}$  is \emph{$R$-linearly convergent}  (or \emph{converges $R$-linearly}) to $\bar{y}$ if $\limsup_{k \to \infty} \lr{\norm{y_{k} -\bar{y}}}^{\frac{1}{k}} <1$.
 For other notation not explicitly defined here, we refer the reader to \cite{BC2017}.

%%%%%%%%%%%%%%%%%%%%%%%%%%%%%%%%%%%%%%%%%%%%%%%%%%%%
%%%%%%%%%%%%%%%%\section{Auxiliary Results}%%%%%%%%%%%%%%%%%%%%%%
%%%%%%%%%%%%%%%%%%%%%%%%%%%%%%%%%%%%%%%%%%%%%%%%%%%%
\section{Auxiliary Results} \label{sec:AuxiliaryResults}
In this section, we collect some definitions and facts mainly from the book \cite{BC2017}, which are fundamental to our analysis later. Moreover, we also show some    auxiliary results to simplify our proofs in subsequent sections.

The following result plays an essential role to prove multiple results in the next sections. 
\begin{fact} {\rm \cite[Lemma~5.31]{BC2017}}  \label{fact:alphakINEQ} 
	Let $(\alpha_{k})_{k \in \mathbb{N}}$, $(\beta_{k})_{k \in \mathbb{N}}$, $(\gamma_{k})_{k \in \mathbb{N}}$, and $(\varepsilon_{k})_{k \in \mathbb{N}}$  be sequences in $\mathbb{R}_{+}$ such that $\sum_{k \in \mathbb{N}} \gamma_{k} < \infty$,  $\sum_{k \in \mathbb{N}} \varepsilon_{k} < \infty$, and
	\begin{align*}
	(\forall k \in \mathbb{N}) \quad \alpha_{k+1} \leq (1+\gamma_{k}) \alpha_{k} - \beta_{k} + \varepsilon_{k}.
	\end{align*}
	Then $\lim_{k \to \infty} \alpha_{k}$ exists in $\mathbb{R}_{+}$   and $\sum_{k \in \mathbb{N}} \beta_{k} < \infty$.
\end{fact}

\subsection{Nonexpansive operators}

All algorithms studied in this work are based on nonexpansive operators. In fact, a lot of iteration mappings in optimization algorithms are  nonexpansive. 
\begin{definition} {\rm \cite[Definition~4.1]{BC2017}} \label{definition:Nonexpansive}
	Let $D$ be a nonempty subset of $\mathcal{H}$ and let $T: D \to \mathcal{H}$. Then $T$ is 
	\begin{enumerate}
		\item \label{definition:Nonexpansive:quasinonex} \index{quasinonexpansive operator} \emph{quasinonexpansive} if $	\lr{\forall x \in D}$ $ \lr{\forall y \in \Fix T}$ $\norm{Tx -y} \leq \norm{x-y}$;
		\item  \label{definition:Nonexpansive:nonexp} \index{nonexpansive operator} \emph{nonexpansive} if it is Lipschitz continuous with constant $1$, i.e., $	(\forall  x \in D)$ $ (\forall y \in D) $ $ \norm{Tx-Ty} \leq \norm{x-y}$;
		\item  \label{definition:Nonexpansive:firmlynonexp} \index{firmly nonexpansive operator} 
		\emph{firmly nonexpansive} if $(\forall  x \in D)$ $ (\forall y \in D)$ $ \norm{Tx-Ty}^{2} +\norm{(\Id -T)x - (\Id -T)y}^{2} \leq \norm{x-y}^{2}$.
	\end{enumerate}
\end{definition}

The following \cref{fact:quasinonexCC} is basic for  the analysis of many results later. 
\begin{fact} {\rm \cite[Proposition~4.23(ii)]{BC2017}} \label{fact:quasinonexCC}
	Let $D$ be a nonempty closed and convex subset of $\mathcal{H}$ and let $T : D \to \mathcal{H}$ be quasinonexpansive. Then $\Fix T$ is closed and convex. 
\end{fact}

Although \cite[Definition~4.33]{BC2017} considers only the case $\alpha \in \left]0,1\right[\,$, we extend the definition to $\alpha \in \left]0,1\right]$. Clearly, by \cref{definition:averaged}, $T$ is $1$-averaged if and only if $T$ is nonexpansive. This extension will facilitate our   statements in the following sections. 
\begin{definition} {\rm \cite[Definition~4.33]{BC2017}} 	\label{definition:averaged}
	\index{averaged nonexpansive operator} 
	Let $D$ be a nonempty subset of $\mathcal{H}$,  let $T: D \to \mathcal{H}$ be nonexpansive, and let $\alpha \in \left]0,1\right]$. Then $T$ is \emph{averaged with constant $\alpha$}, or  \emph{$\alpha$-averaged}, if there exists a nonexpansive operator $R: D \to \mathcal{H}$ such that $T=(1-\alpha) \Id +\alpha R$.	
\end{definition}	

Clearly, by \Cref{definition:Nonexpansive,definition:averaged}, firmly nonexpansive operators and averaged operators must be nonexpansive operators. 
\begin{fact} {\rm  \cite[Remark~4.34(iii)]{BC2017} } \label{fact:firmlynonexpansiveaveraged}
	Let $D$ be a nonempty subset of $\mathcal{H}$ and  let $T: D \to \mathcal{H}$. Then $T$ is firmly nonexpansive if and only if it is $\frac{1}{2}$-averaged.  
\end{fact}

In fact,  \cref{fact:Averagedlambdaalpha} extends \cite[Proposition~4.40]{BC2017}  from $\alpha \in \left]0,1\right[$ and $\lambda \in \left]0,\frac{1}{\alpha}\right[$ to $\alpha \in \left]0,1\right]$ and $\lambda \in \left]0,\frac{1}{\alpha}\right]$, but it is not difficult to see that the extension is trivial by  \cref{definition:averaged}.
\begin{fact} {\rm  \cite[Proposition~4.40]{BC2017} } \label{fact:Averagedlambdaalpha}
	Let $D$ be a nonempty subset of $\mathcal{H}$,  let $T: D \to \mathcal{H}$,  let $\alpha \in \left]0,1\right]$, and let $\lambda \in \left]0,\frac{1}{\alpha}\right]$. Then $T$ is $\alpha$-averaged if and only if $(1-\lambda) \Id +\lambda T$ is $\lambda\alpha$-averaged. 
\end{fact}

\cref{lemma:yxzx}\cref{lemma:yxzx:norm}$\&$\cref{lemma:yxzx:lambdaalpha}  are inspired by \cite[Theorem~5.15 and Proposition~5.34]{BC2017} which work on the weak convergence of the 
 Krasnosel'ski\v{\i}-Mann Iterations. Specifically, the idea of \cref{lemma:yxzx}\cref{lemma:yxzx:norm}$\&$\cref{lemma:yxzx:lambdaalpha} is used in the proofs \cite[Theorem~5.15 and Proposition~5.34]{BC2017}.

\begin{proposition} \label{lemma:yxzx}
	Let $\alpha \in \left]0,1\right]$ and let $T:\mathcal{H} \to \mathcal{H}$ be an $\alpha$-averaged operator with $\Fix T \neq \varnothing$. Let $  x $ and $e  $ be in $ \mathcal{H}$ and  let $\lambda $ 
	and   $\eta $ be in $\mathbb{R}_{+}$. Define
	\begin{align} \label{eq:lemma:yxzx:yzx}
	y_{x} := (1-\lambda) x+\lambda Tx \quad \text{and} \quad 
 z_{x} :=   (1-\lambda) x+\lambda Tx +\eta e=y_{x} +\eta e.
	\end{align}
	Then the   following statements hold.
	\begin{enumerate}
		\item \label{lemma:yxzx:R}  There exists a nonexpansive operator $R:\mathcal{H} \to \mathcal{H}$ such that $\Fix T=\Fix R$, $T-\Id = \alpha (R -\Id)$, and $y_{x}=(1-\alpha \lambda)x+\alpha \lambda Rx$.
		\item  \label{lemma:yxzx:norm} Let  $ \bar{x} \in \Fix T$. Then 
		\begin{align*}
		&\norm{y_{x} -\bar{x}}^{2} \leq \norm{x -\bar{x}}^{2} -\lambda \lr{\frac{1}{\alpha} -\lambda} \norm{x -Tx}^{2};\\
		& \norm{z_{x} -\bar{x}}^{2}  
		\leq  \norm{x -\bar{x}}^{2} -\lambda  \lr{\frac{1}{\alpha} -\lambda} \norm{x -Tx}^{2} +\eta \norm{e} \lr{2\norm{y_{x} -\bar{x}} +\eta \norm{e}}.
		\end{align*}
		\item \label{lemma:yxzx:lambdaalpha} Let $ \bar{x} \in \Fix T$ and   $\lambda \in \left[0, \frac{1}{\alpha}\right]$. Then
		\begin{align*}
		&\norm{y_{x} -\bar{x}} \leq \norm{x -\bar{x}};\\
		&\norm{z_{x} -\bar{x}} \leq \norm{y_{x} -\bar{x}} +\eta \norm{e} \leq \norm{x-\bar{x}} +\eta \norm{e};\\
		&\norm{Ty_{x}-y_{x}} \leq \norm{Tx-x}.
		\end{align*} 
	\end{enumerate}
\end{proposition}

\begin{proof}
	\cref{lemma:yxzx:R}: Because $T$ is $\alpha$-averaged, by \cref{definition:averaged}, there exists a nonexpansive operator $R:\mathcal{H} \to \mathcal{H}$ such that 
	\begin{align*}
	T =(1-\alpha) \Id +\alpha R \Leftrightarrow T-\Id = \alpha (R -\Id).
	\end{align*}
	Because $\alpha \neq 0$, we observe that for every $u \in \mathcal{H}$,
	\begin{align*}
	u \in \Fix T \Leftrightarrow u=Tu=(1-\alpha) u +\alpha Ru \Leftrightarrow u = Ru \Leftrightarrow u \in \Fix R,
	\end{align*}
	which leads to $\Fix T =\Fix R$.	
	Moreover, it is clear that
	\begin{align*}
	y_{x} =x + \lambda (Tx-x) = x +\alpha \lambda (Rx -x) = (1-\alpha \lambda)x+\alpha \lambda Rx.
	\end{align*}

	\cref{lemma:yxzx:norm}:  Apply \cref{lemma:yxzx:R} in the first equality and the inequality below, and employ \cite[Corollary~2.15]{BC2017}  in the third equality  below to ensure that 
	\begin{align*}
	\norm{y_{x} -\bar{x}}^{2} =& \norm{(1-\alpha \lambda)x+\alpha \lambda Rx-\bar{x}}^{2}\\
	=&\norm{ (1-\alpha \lambda) \lr{x -\bar{x}} +\alpha \lambda \lr{Rx -\bar{x} } }^{2}\\
	=&	 (1-\alpha \lambda) \norm{ x -\bar{x}}^{2} + \alpha \lambda \norm{Rx -\bar{x}}^{2} -\alpha \lambda \lr{1-\alpha \lambda} \norm{x -Rx}^{2}\\
	\leq & (1-\alpha \lambda) \norm{ x -\bar{x}}^{2} + \alpha \lambda \norm{ x -\bar{x}}^{2} -\alpha \lambda \lr{1-\alpha \lambda} \frac{1}{\alpha^{2}}\norm{x -Tx}^{2}\\
	=& \norm{ x -\bar{x}}^{2} - \lambda \lr{ \frac{1}{\alpha}-  \lambda} \norm{x -Tx}^{2}.
	\end{align*}
	Based on \cref{eq:lemma:yxzx:yzx} and the inequality proved above, we establish that 
	\begin{align*}
	\norm{z_{x} -\bar{x}}^{2}  =& \norm{y_{x}  -\bar{x} +\eta e}^{2}\\
	\leq & \norm{ y_{x}  -\bar{x}  }^{2}+\eta \norm{e} \lr{2\norm{y_{x} -\bar{x}} +\eta \norm{e}}\\
	\leq &\norm{x -\bar{x}}^{2} -\lambda  \lr{\frac{1}{\alpha} -\lambda} \norm{x -Tx}^{2} +\eta \norm{e} \lr{2\norm{y_{x} -\bar{x}} +\eta \norm{e}},
	\end{align*}
	where in the first inequality above we use the fact 
	\begin{align*}
	(\forall u \in \mathcal{H}) (\forall v \in \mathcal{H}) \quad \norm{u+v}^{2} \leq \norm{u}^{2} + \norm{v} ( 2\norm{u} +\norm{v}).
	\end{align*}

	\cref{lemma:yxzx:lambdaalpha}: Note that \cref{lemma:yxzx:norm} necessitates
	\begin{align*}
	\norm{y_{x} -\bar{x}} \leq \norm{x -\bar{x}},
	\end{align*}
	which, connected with \cref{eq:lemma:yxzx:yzx}, guarantees that 
	\begin{align*}
	\norm{z_{x} -\bar{x}} \leq \norm{y_{x} -\bar{x}} +\eta \norm{e} \leq \norm{x-\bar{x}} +\eta \norm{e}.
	\end{align*}
	In addition, applying $\triangle$-inequality and $\alpha \lambda \in \left[0,1\right]$ in the first inequality and utilizing \cref{lemma:yxzx:R} and the nonexpansiveness of $R$ in the second inequality below, we obtain that
	\begin{align*}
	\norm{Ty_{x}-y_{x}} ~\stackrel{\text{\cref{lemma:yxzx:R}}}{=}~ & \alpha \norm{ Ry_{x}-y_{x} } 
	\\
	~\stackrel{\text{\cref{lemma:yxzx:R}}}{=}~ & \alpha \norm{ Ry_{x} -Rx +Rx - (1-\alpha \lambda)x -\alpha \lambda Rx }\\
	~=~& \alpha \norm{ Ry_{x} -Rx  +(1-\alpha \lambda) \lr{Rx-x} }\\
	~\leq~ & \alpha \lr{  \norm{Ry_{x} -Rx} + (1-\alpha \lambda) \norm{Rx-x}  } 
	\\
	~\leq~ & \alpha \lr{  \norm{ y_{x} - x} + \frac{1}{\alpha}(1-\alpha \lambda) \norm{Tx-x}  } \quad 
	\\
	\stackrel{\cref{eq:lemma:yxzx:yzx}}{=} & \alpha \lambda \norm{Tx-x} + (1-\alpha \lambda) \norm{Tx-x} 
	\\
	~=~&\norm{Tx-x} .
	\end{align*}
Altogether, the proof is complete.
\end{proof}

\subsection{Resolvent of monotone operators}

\begin{definition} {\rm \cite[Definitions~20.1 and 20.20]{BC2017}} \label{definition:maximallymonotone}
	Let $ A : \mathcal{H} \to 2^{\mathcal{H}}$ be a set-valued operator. Then $A$ is  \emph{monotone} \index{monotone} if $	(\forall (x,u) \in \gra A ) $ $ (\forall  (y,v) \in \gra A) $ $ \innp{x-y, u-v} \geq 0$.  We say a monotone operator $A$ is \emph{maximally monotone} \index{monotone!maximally monotone} (or \emph{maximal monotone}) if there exists no monotone operator $ B :\mathcal{H} \to 2^{\mathcal{H}}$ such that $\gra B$ properly contains $\gra A$, i.e., for every $(x,u) \in \mathcal{H} \times \mathcal{H}$, 
	\begin{align*}
	(x,u) \in \gra A \Leftrightarrow \lr{ \forall  (y,v) \in \gra A } \innp{x-y, u-v} \geq 0.
	\end{align*}
\end{definition}

\begin{definition} {\rm \cite[Definition~23.1]{BC2017}} \label{defn:ResolventApproxi}
	Let $A: \mathcal{H} \to 2^{\mathcal{H}}$ and let $\gamma \in \mathbb{R}_{++}$. The \emph{resolvent of $A$} \index{resolvent} is 
	\begin{align*}
	\J_{A} = (\Id + A)^{-1}.
	\end{align*}
\end{definition}

The following \cref{fact:cAMaximallymonotone} demonstrates that the resolvent of a maximally monotone operator is single-valued, full domain, and firmly nonexpansive, which is fundamental for our study on generalized proximal point algorithms in this work. 

\begin{fact}  {\rm \cite[Proposition~23.10]{BC2017}} \label{fact:cAMaximallymonotone}
	Let  $A: \mathcal{H} \to 2^{\mathcal{H}}$ be such that $\dom A \neq \varnothing$, set $D:= \Range (\Id +A)$, and set $T=\J_{A } |_{ D}$. Then $A$ is maximally monotone if and only if $T$ is firmly nonexpansive and $D =\mathcal{H}$.
\end{fact}

\cref{fact:FixJcAzerA} will be used frequently in our proofs later.
\begin{fact} {\rm \cite[Proposition~23.38]{BC2017}} \label{fact:FixJcAzerA}
	Let $A: \mathcal{H} \to 2^{\mathcal{H}}$ be monotone and  let $ \gamma \in \mathbb{R}_{++}$. Then $\Fix \J_{\gamma A} = \zer A$.
\end{fact}

\begin{lemma} \label{lemma:resolvents:yxzx}
	Let $A: \mathcal{H} \to 2^{\mathcal{H}}$ be maximally monotone with $\zer A \neq \varnothing$, let $ x$ and $e $ be in $\mathcal{H}$,  let 
	$\gamma \in \mathbb{R}_{++}$,  and  let $\lambda$ and   $\eta$ be in $ \mathbb{R}_{+}$. Define
	\begin{align*}
	y_{x} := (1-\lambda) x+\lambda \J_{\gamma A}x  \quad \text{and} \quad 
	z_{x} :=   (1-\lambda) x+\lambda \J_{\gamma A}x +\eta e = y_{x} +\eta e. 
	\end{align*}
	Then for every $ \bar{x} \in \zer A$,  
	\begin{align*}
	&\norm{y_{x} -\bar{x}}^{2} \leq \norm{x -\bar{x}}^{2} -\lambda \lr{2 -\lambda} \norm{x -\J_{\gamma A}x}^{2};\\
	&\norm{z_{x} -\bar{x}}^{2} \leq \norm{x -\bar{x}}^{2} -\lambda  \lr{2 -\lambda} \norm{x -\J_{\gamma A}x}^{2} +\eta \norm{e} \lr{2\norm{y_{x} -\bar{x}} +\eta \norm{e}}.
	\end{align*}
\end{lemma}

\begin{proof}
	Because $A$ is maximally monotone and $\gamma \in \mathbb{R}_{++}$, applying \cite[Proposition~20.22]{BC2017}   with $u=0$ and $z=0$, we know that $\gamma A$ is maximally monotone. Then, via \Cref{fact:cAMaximallymonotone,fact:firmlynonexpansiveaveraged},  we know that $ \J_{\gamma A}: \mathcal{H} \to \mathcal{H}$ is $\frac{1}{2}$-averaged. Moreover, due to \cref{fact:FixJcAzerA}, $ \Fix \J_{\gamma A} = \zer A $. Hence, the desired results follow immediately from
	\cref{lemma:yxzx}\cref{lemma:yxzx:norm}.
\end{proof}

\cref{fact:MaximallyMonotoneCC} below is fundamental to our analysis in some results below. 
\begin{fact} 	{\rm \cite[Proposition~23.39]{BC2017}} 
	\label{fact:MaximallyMonotoneCC}
	Let $A: \mathcal{H} \to 2^{\mathcal{H}}$ be maximally monotone. Then $\zer A $ is closed and convex. 
\end{fact}

\begin{corollary} \label{corollary:JcA}
	Let  $A: \mathcal{H} \to 2^{\mathcal{H}}$ be maximally monotone, let $x \in \mathcal{H}$, and let $\gamma \in \mathbb{R}_{++}$. Then the following hold.
	\begin{enumerate}
		\item \label{corollary:JcA:gra}   $\lr{ \J_{\gamma A} x , \frac{1}{\gamma} \lr{x-\J_{\gamma A} x}}   \in \gra A$.
		\item  \label{corollary:JcA:mulambda}  $\lr{ \forall \mu \in \mathbb{R}_{++}}$  $ \J_{\gamma A}(x ) = \J_{\mu A} \left( \frac{\mu}{\gamma}x + \left( 1- \frac{\mu}{\gamma} \right)  \J_{\gamma A}x  \right)$.
	\end{enumerate} 
\end{corollary}

\begin{proof}
	\cref{corollary:JcA:gra}: According to \cref{fact:cAMaximallymonotone}, we know that 
	$\J_{\gamma A} : \mathcal{H} \to \mathcal{H}$ is full domain. Hence,
	The desired result follows immediately from \cite[Proposition~23.2(ii)]{BC2017}.
	
	\cref{corollary:JcA:mulambda}: Apply \cite[Proposition~23.31(i)]{BC2017} with $\lambda =\frac{\mu}{\gamma}$ to deduce the required result. 
\end{proof}

\begin{lemma}\label{lemma:JGammaAFix}
	Let $A: \mathcal{H} \to 2^{\mathcal{H}}$ be maximally monotone with $\zer A \neq \varnothing$ and  let $ \gamma \in \mathbb{R}_{++}$. Then
	\begin{align*}
	(\forall x \in \mathcal{H})	(\forall z \in \zer A) \quad \norm{\J_{\gamma A} x -z}^{2} +\norm{ \lr{\Id - \J_{\gamma A}}x}^{2} \leq \norm{x -z}^{2}.
	\end{align*}
\end{lemma}

\begin{proof}
	As a consequence of \cref{fact:cAMaximallymonotone}, $\J_{\gamma A}$ is firmly nonexpansive and full domain. Moreover, due to \cref{fact:FixJcAzerA}, $ \Fix \J_{\gamma A} = \zer A$. Therefore, 
	the desired inequality follows easily from   \cref{definition:Nonexpansive}\cref{definition:Nonexpansive:firmlynonexp}.
\end{proof}

The following result complements \cite[Proposition~2.16]{OuyangWeakStrongGPPA2021} which also provides sufficient conditions for the inclusion given in \cref{lemma:ykweakcluster}. 
Notice that the assumption $\inf_{k \in \mathbb{N}} c_{k} >0$ is critical in \cite[Proposition~2.16]{OuyangWeakStrongGPPA2021}.
\begin{lemma} \label{lemma:ykweakcluster}
	Let $A :\mathcal{H} \to 2^{\mathcal{H}}$ be maximally monotone with $\zer A \neq \varnothing$.
	Let $(y_{k})_{k \in \mathbb{N}}$ be in $\mathcal{H}$ and let $(c_{k})_{k \in \mathbb{N}}$ be in $\mathbb{R}_{++}$. Denote the set of all weak sequential clusters of the sequence $(y_{k})_{k \in \mathbb{N}}$ by $\Omega \lr{ \lr{y_{k}}_{k \in \mathbb{N}} }$.
	Suppose that $\frac{1}{c_{k}} \lr{y_{k} -\J_{c_{k} A}y_{k}} \to 0 $ and that one of the following assumptions hold. 
	\begin{itemize}
		\item[{\rm (A1)}]   \label{lemma:ykweakcluster:supC} $\bar{c} := \sup_{k \in \mathbb{N}} c_{k} <\infty$.
		\item[{\rm (A2)}]  \label{lemma:ykweakcluster:yk+1} $(\forall k \in \mathbb{N})$ $y_{k+1} = \J_{c_{k} A}y_{k}$.
	\end{itemize}
	Then $\Omega \lr{ \lr{y_{k}}_{k \in \mathbb{N}} } \subseteq \zer A$.
\end{lemma}

\begin{proof}
	If $\Omega \lr{ \lr{y_{k}}_{k \in \mathbb{N}} }=\varnothing$, then the required result is trivial. Suppose that $\Omega \lr{ \lr{y_{k}}_{k \in \mathbb{N}} } \neq \varnothing$.
	Let $\bar{y} \in \Omega \lr{ \lr{y_{k}}_{k \in \mathbb{N}} }$. Then there exists a subsequence $ \lr{y_{k_{i}}}_{i \in \mathbb{N}}$ of $ \lr{y_{k}}_{k \in \mathbb{N}}$ such that 
	\begin{align} \label{eq:lemma:ykweakcluster}
	y_{k_{i}} \weakly \bar{y}.
	\end{align}
	
	\emph{Case~1}: Suppose that the assumption (A1) holds. 	As a consequence of  $\frac{1}{c_{k}} \lr{y_{k} -\J_{c_{k} A}y_{k}} \to 0 $, we know that 
	\begin{align*}
	\norm{y_{k} - \J_{c_{k} A}y_{k}} = c_{k} \norm{\frac{1}{c_{k}} \lr{y_{k} -\J_{c_{k} A}y_{k}}  } \leq \bar{c} \norm{\frac{1}{c_{k}} \lr{y_{k} -\J_{c_{k} A}y_{k}}  } \to 0.
	\end{align*}
	Combine this with  \cref{eq:lemma:ykweakcluster} to establish that 
	\begin{align} \label{eq:lemma:ykweakcluster:case1}
	\J_{c_{k_{i}} A}y_{k_{i}} =\lr{ \J_{c_{k_{i}} A}y_{k_{i}} - y_{k_{i}} } +y_{k_{i}} \weakly \bar{y}
	\end{align}
	Because $A$ is monotone and, via \cref{corollary:JcA}\cref{corollary:JcA:gra}, 
	\begin{align*}
	(\forall i \in \mathbb{N})\quad \lr{ \J_{c_{k_{i}} A}y_{k_{i}} , \frac{1}{c_{k_{i}}} \lr{ y_{k_{i}}-  \J_{c_{k_{i}} A}y_{k_{i}} }} \in \gra A,
	\end{align*}
	we observe that 
	\begin{align*}
	\lr{\forall \lr{x,u} \in \gra A} \quad 	\innp{ \J_{c_{k_{i}} A}y_{k_{i}} -x, \frac{1}{c_{k_{i}}} \lr{ y_{k_{i}}-  \J_{c_{k_{i}} A}y_{k_{i}} } -u } \geq 0,
	\end{align*} 
	which, connected with \cref{eq:lemma:ykweakcluster:case1} and the assumption   $\frac{1}{c_{k}} \lr{y_{k} -\J_{c_{k} A}y_{k}} \to 0 $, ensures that 
	\begin{align*}
	\lr{\forall (x,u) \in \gra A} \quad \innp{\bar{y}-x, 0-u} \geq 0.
	\end{align*}
	This, due to \cref{definition:maximallymonotone}, guarantees that $\lr{\bar{y}, 0} \in \gra A$, that is, $\bar{y} \in \zer A$. 
	
	\emph{Case~2}:  Suppose that (A2) holds. Then 
	\begin{align*}
	\J_{c_{k_{i} -1} A}y_{k_{i} -1} = y_{k_{i}}  \weakly \bar{y}.
	\end{align*}
	\cref{corollary:JcA}\cref{corollary:JcA:gra} leads to $(\forall i \in \mathbb{N})$ $\lr{  \J_{c_{k_{i} -1} A}y_{k_{i} -1}, \frac{1}{c_{k_{i} -1}} \lr{ y_{k_{i} -1} -   \J_{c_{k_{i} -1} A}y_{k_{i} -1}} } \in \gra A$. Combine this with the monotonicity of $A$ to deduce that
	\begin{align*}
	\lr{\forall \lr{x,u} \in \gra A} \quad 	\innp{  \J_{c_{k_{i} -1} A}y_{k_{i} -1} -x, \frac{1}{c_{k_{i} -1}} \lr{ y_{k_{i} -1}  -   \J_{c_{k_{i} -1} A}y_{k_{i} -1} } -u } \geq 0,
	\end{align*} 
	which, connected with 
	\begin{align*}
	\J_{c_{k_{i} -1} A}y_{k_{i} -1}  \weakly \bar{y} \quad \text{and} \quad \frac{1}{c_{k}} \lr{y_{k} -\J_{c_{k} A}y_{k}} \to 0,
	\end{align*}
	necessitates that 
	\begin{align*}
	\lr{\forall (x,u) \in \gra A} \quad \innp{\bar{y}-x, 0-u} \geq 0.
	\end{align*}
	As we pointed out in the proof of Case~1 above, this entails  $\bar{y} \in \zer A$. 
	
	Altogether, the proof is complete. 
\end{proof}

\subsection{Metrical subregularity}
\begin{definition} \label{definition:metricsubregularity:subregular}  {\rm \cite[Pages~183 and 184]{DontchevRockafellar2014}} 
	Let $F: \mathcal{H} \to 2^{\mathcal{H}}$ be a set-valued operator. $F$ is called 
	 \emph{metrically subregular at $\bar{x}$ for $\bar{y}$} if $\lr{ \bar{x}, \bar{y} } \in \gra F$ and there exists $\kappa \in \mathbb{R}_{+}$ along with a neighborhood $U$ of $\bar{x}$ such that  
	\begin{align*}
	(\forall x \in U) \quad	\dist \lr{ x, F^{-1} (\bar{y})} \leq \kappa \dist \lr{\bar{y}, F(x)}.
	\end{align*}
	The constant $\kappa$ is called \emph{constant of metric subregularity}.
	The infimum of all $\kappa$ for which the inequality above holds is the \emph{modulus of metric subregularity}, denoted by $\text{subreg}\lr{F;\bar{x}|\bar{y}}$. The absence of metric subregularity is signaled by 
	$\text{subreg}\lr{F;\bar{x}|\bar{y}}  =\infty$.
\end{definition}

\cref{proposition:TFMetricSubreg}\cref{proposition:TFMetricSubreg:MSubR}  is inspired by  \cite[Lemma~3.8]{BauschkeNollPhan2015}. In fact, \cref{proposition:TFMetricSubreg}\cref{proposition:TFMetricSubreg:MSubR} improves \cite[Lemma~3.8]{BauschkeNollPhan2015} from the following three aspects: it replaces the $\alpha$-averaged operator $T$ in  \cite[Lemma~3.8]{BauschkeNollPhan2015} by the affine combination $F_{\lambda} =\lr{1-\lambda}\Id +\lambda T$ below;
it generalizes \cite[Lemma~3.8]{BauschkeNollPhan2015} from boundedly linearly regularity (see \cite[Definition~2.1]{BauschkeNollPhan2015} for a detailed definition) to metrical  subregularity; and    the upper bound  in \cref{proposition:TFMetricSubreg}\cref{proposition:TFMetricSubreg:B}    is better than the corresponding one in the inequality $(21)$ of \cite[Lemma~3.8]{BauschkeNollPhan2015}.

 \cref{proposition:TFMetricSubreg}  will play a critical role to prove 
\cref{theorem:KMI} below.  
\begin{proposition} \label{proposition:TFMetricSubreg}
	Let $T: \mathcal{H} \to \mathcal{H}$ be $\alpha$-averaged with $\alpha \in \left]0,1\right]$ and $\Fix T \neq \varnothing$. 	Let $\lambda \in \left]0,\frac{1}{\alpha}\right[\,$.  Define 
	\begin{align*}
	F_{\lambda} :=\lr{1-\lambda}\Id +\lambda T.
	\end{align*}
	Then the following statements hold.
	\begin{enumerate}
		\item \label{proposition:TFMetricSubreg:H} For every $x \in \mathcal{H}$ and every $z \in \Fix T$, we have that
		\begin{align*}
		&\norm{F_{\lambda}x-z}^{2} + \frac{\lambda \lr{1-\lambda \alpha}}{ \alpha} \norm{\lr{\Id - T}x}^{2} \leq \norm{x-z}^{2};\\
		&\frac{\lambda \lr{1-\lambda \alpha}}{ \alpha} \norm{\lr{\Id - T}x}^{2} \leq \dist^{2} \lr{x, \Fix T} - \dist^{2} \lr{F_{\lambda}x, \Fix T}.
		\end{align*} 
		Consequently, $(\forall x \in \mathcal{H})$ $  \dist  \lr{F_{\lambda}x, \Fix T} \leq \dist  \lr{x, \Fix T} $.
		\item \label{proposition:TFMetricSubreg:MSubR} Suppose that $\Id -T$ is metrically subregular at $\bar{x} \in \Fix T$ for $0 \in \lr{\Id -T}\bar{x}$, i.e., 
		\begin{align} \label{eq:proposition:TFMetricSubreg:MetricSub} 
		(\exists \kappa >0) (\exists \delta >0) (\forall x \in B[\bar{x}; \delta]) \quad \dist \lr{x, \Fix T} \leq \kappa \norm{x - Tx}.
		\end{align}
		Define $\rho:=\lr{  1- \frac{\lambda \lr{1-\lambda \alpha}}{ \alpha\kappa^{2}} }^{\frac{1}{2}}$. Then the following hold. 
		\begin{enumerate}
			\item  \label{proposition:TFMetricSubreg:rho} $\rho \in \left[ 0,	\frac{1}{ \lr{ 1+ \frac{\lambda \lr{1-\lambda \alpha}}{ \alpha\kappa^{2}} }^{\frac{1}{2}}   } \right[ \subseteq  \left[ 0,1\right[\,$.
			
			\item \label{proposition:TFMetricSubreg:B} $(\forall x \in B[\bar{x}; \delta])$  $  \dist  \lr{F_{\lambda}x, \Fix T} \leq \rho \dist  \lr{x, \Fix T} $.
		\end{enumerate}
		
	\end{enumerate}
\end{proposition}

\begin{proof}
	\cref{proposition:TFMetricSubreg:H}:	Because $T: \mathcal{H} \to \mathcal{H}$ is $\alpha$-averaged, due to \cref{fact:Averagedlambdaalpha}, we know that $F_{\lambda}$ is $\lambda\alpha$-averaged with $\lambda \alpha \in \left]0,1\right[\,$.  As a consequence of  $\lambda \neq 0$ and \cref{fact:quasinonexCC}, we have that $\Fix F_{\lambda} =\Fix T $ is nonempty closed and convex. In view  of the definition of $F_{\lambda}$, we get that 
	\begin{align} \label{eq:proposition:TFMetricSubreg:IdF}
	\Id - F_{\lambda} = \lambda \lr{\Id - T}.
	\end{align}
	Let  $x \in \mathcal{H}$.  Based on \cite[Proposition~4.35]{BC2017} and $\Fix F_{\lambda} =\Fix T $, we observe that for every $z \in \Fix T$,
	\begin{align*}
	&	\norm{F_{\lambda}x-z}^{2} + \frac{1-\lambda \alpha}{\lambda \alpha} \norm{\lr{\Id - F_{\lambda}}x}^{2} \leq \norm{x-z}^{2}\\
	\stackrel{\cref{eq:proposition:TFMetricSubreg:IdF}}{\Leftrightarrow} & \norm{F_{\lambda}x-z}^{2} + \frac{\lambda \lr{1-\lambda \alpha}}{ \alpha} \norm{\lr{\Id - T}x}^{2} \leq \norm{x-z}^{2},
	\end{align*}
	which, noticing that $\Pro_{\Fix T}x \in \Fix T$,  implies that, 
	\begin{align*}
	&	\frac{\lambda \lr{1-\lambda \alpha}}{ \alpha} \norm{\lr{\Id - T}x}^{2} \leq \norm{x-\Pro_{\Fix T}x}^{2} - \norm{F_{\lambda}x-\Pro_{\Fix T}x}^{2}\\
	\Rightarrow & \frac{\lambda \lr{1-\lambda \alpha}}{ \alpha} \norm{\lr{\Id - T}x}^{2} \leq \norm{x-\Pro_{\Fix T}x}^{2} - \norm{F_{\lambda}x-\Pro_{\Fix T}\lr{F_{\lambda}x}}^{2}\\
	\Leftrightarrow & \frac{\lambda \lr{1-\lambda \alpha}}{ \alpha} \norm{\lr{\Id - T}x}^{2} \leq \dist^{2} \lr{x, \Fix T} - \dist^{2} \lr{F_{\lambda}x, \Fix T}.
	\end{align*}
	
	\cref{proposition:TFMetricSubreg:MSubR}:  Let $x \in B[\bar{x}; \delta]$. Set $\eta:= \frac{\lambda \lr{1-\lambda \alpha}}{ \alpha} \in \mathbb{R}_{++}$.  According to \cref{eq:proposition:TFMetricSubreg:MetricSub}, we have that
	\begin{align*}
	 \dist^{2} \lr{x, \Fix T} \leq  \kappa^{2} \norm{x - Tx}^{2} \stackrel{\text{\cref{proposition:TFMetricSubreg:H}}}{\leq}  \frac{\kappa^{2}}{\eta } \lr{ \dist^{2} \lr{x, \Fix T} - \dist^{2} \lr{F_{\lambda}x, \Fix T} }.
	\end{align*} 
This necessitates  that 
		\begin{align} \label{eq:proposition:TFMetricSubreg:B:-}
	 \dist^{2} \lr{ F_{\lambda}x, \Fix T } \leq \lr{1 -\frac{\eta }{\kappa^{2}} } \dist^{2} \lr{x, \Fix T},  
		\end{align}
which forces that 
	$1 -\frac{\eta}{\kappa^{2}}  \in \mathbb{R}_{+}$. Notice that $	1 -\frac{\eta}{\kappa^{2}} < \frac{1}{1+ \frac{\eta}{\kappa^{2}} } \Leftrightarrow 1 -\frac{\eta^{2}}{\kappa^{4}} < 1$.
	 So, we establish that $0 \leq  1- \frac{\lambda \lr{1-\lambda \alpha}}{ \alpha\kappa^{2}} <  \frac{1}{ 1+ \frac{\lambda \lr{1-\lambda \alpha}}{ \alpha\kappa^{2}} } <1$, that leads to \cref{proposition:TFMetricSubreg:rho}. 
	
	In addition, \cref{eq:proposition:TFMetricSubreg:B:-} clearly ensures the required inequality in \cref{proposition:TFMetricSubreg:B}.	
\end{proof}

The following \cref{lemma:metricallysubregularEQ}\cref{lemma:metricallysubregularEQ:EQ} can also be obtained by substituting $\mathcal{B} =0$ in \cite[Lemma~3.3]{ShenPan2016}. The idea of the following proof is almost the same as that of \cite[Lemma~3.3]{ShenPan2016}. For completeness and convenience of later references, we attach detailed results and proofs below. 

\begin{lemma} \label{lemma:metricallysubregularEQ}
	Let  $A: \mathcal{H} \to 2^{\mathcal{H}}$ be   maximally monotone with $\zer A \neq \varnothing$, let $\bar{x} \in \zer A$, and let $\gamma \in \mathbb{R}_{++}$.    Then the following hold. 
	\begin{enumerate}
		\item  \label{lemma:metricallysubregularEQ:A} Suppose that $A$ is metrically subregular at $\bar{x}$ for $0 \in A\bar{x}$, i.e., 
		\begin{align} \label{eq:eqlemma:metricallysubregularEQ:A}
		(\exists \kappa >0) (\exists \delta >0) (\forall x \in B[\bar{x}; \delta]) \quad \dist \lr{x, A^{-1}0} \leq \kappa \dist \lr{0, Ax}.
		\end{align}
		Then $  \Id -\J_{\gamma A}  $ is metrically subregular at $\bar{x}$ for $0= \lr{\Id -\J_{\gamma A} } \bar{x}$; more precisely, 
		\begin{align*}
	(\forall x \in B[\bar{x}; \delta]) \quad 	\dist \lr{x, \lr{\Id -\J_{\gamma A}}^{-1}0} \leq \lr{1 + \frac{\kappa}{\gamma}} \dist \lr{0,  \lr{\Id -\J_{\gamma A}} x}.
		\end{align*}
		
		\item  \label{lemma:metricallysubregularEQ:IdJA} Suppose that $  \Id -\J_{\gamma A} $ is metrically subregular at $\bar{x}$ for $0 = \lr{\Id -\J_{\gamma A}}\bar{x}$, i.e.,  
		\begin{align} \label{eq:lemma:metricallysubregularEQ:ID'}
	(\exists \kappa' >0) (\exists \delta' >0) (\forall x \in B[\bar{x}; \delta'])  \quad 	\dist \lr{x, \lr{\Id -\J_{\gamma A}}^{-1}0} \leq  \kappa' \dist \lr{0,  \lr{\Id -\J_{\gamma A}}x}.
		\end{align}
		Then $A$ is metrically subregular at $\bar{x}$ for $0 \in A\bar{x}$; more precisely, 
		\begin{align*}
		(\forall x \in B[\bar{x}; \delta']) \quad \dist \lr{x, A^{-1}0} \leq \kappa'  \gamma \dist \lr{0, Ax}.
		\end{align*}
		
		\item  \label{lemma:metricallysubregularEQ:EQ} $A$ is metrically subregular at $\bar{x}$ for $0 \in A\bar{x}$ if and only if $  \Id -\J_{\gamma A}  $ is metrically subregular at $\bar{x}$ for $0 = \lr{\Id -\J_{\gamma A}} \bar{x}$.
\end{enumerate}

\end{lemma}

\begin{proof}
	As a consequence of \cref{fact:FixJcAzerA}, $\zer \lr{ \Id -\J_{\gamma A}} = \Fix \J_{\gamma A} = \zer A$, which necessitates that 
	\begin{align} \label{eq:lemma:metricallysubregularEQ}
	(\forall x \in \mathcal{H}) \quad \dist \lr{x,\lr{\Id -\J_{\gamma A}}^{-1}0 } = \dist \lr{x, A^{-1}0}.
	\end{align}
	
	\cref{lemma:metricallysubregularEQ:A}: Let $x \in B[\bar{x}; \delta]$.  Employing $\bar{x} \in  \zer A$ and  \cref{lemma:JGammaAFix}, we get 
 that    $\J_{\gamma A}x \in B[\bar{x}; \delta]$. Then applying \cref{eq:eqlemma:metricallysubregularEQ:A} with $x$ replaced by $\J_{\gamma A}x $  in the following first inequality and employing \cref{corollary:JcA}\cref{corollary:JcA:gra}   in the second one, we establish that
	\begin{align} \label{eq:lemma:metricallysubregularEQ:A:ineq}
	\dist \lr{ \J_{\gamma A}x  , A^{-1}0} \leq \kappa \dist \lr{0, A\lr{\J_{\gamma A}x }} \leq \frac{\kappa}{\gamma}\norm{x -\J_{\gamma A}x }.
	\end{align}
	Now,
	\begin{align*}
	\dist \lr{x,\lr{\Id -\J_{\gamma A}}^{-1}0 } \stackrel{\cref{eq:lemma:metricallysubregularEQ}}{=} & \dist \lr{x, A^{-1}0}\\
	\leq~ & \norm{x - \J_{\gamma A}x } + \dist \lr{\J_{\gamma A}x  , A^{-1}0} \\
	\stackrel{\cref{eq:lemma:metricallysubregularEQ:A:ineq}}{\leq} & \norm{x - \J_{\gamma A}x } +  \frac{\kappa}{\gamma}\norm{x -\J_{\gamma A}x }\\
	=~ & \lr{ 1 +\frac{\kappa}{\gamma} } \dist \lr{0, \lr{\Id -\J_{\gamma A}} x}.
	\end{align*}
	
	\cref{lemma:metricallysubregularEQ:IdJA}: Let $x \in B[\bar{x}; \delta']$. 
	Utilize \cref{corollary:JcA}\cref{corollary:JcA:gra}  in the last inequality to obtain that 
	\begin{align*}
	\dist \lr{x, A^{-1}0}  \stackrel{\cref{eq:lemma:metricallysubregularEQ}}{=}    \dist \lr{x,\lr{\Id -\J_{\gamma A}}^{-1}0 }  
	\stackrel{\cref{eq:lemma:metricallysubregularEQ:ID'}}{\leq}    \kappa' \dist \lr{0,  \lr{\Id -\J_{\gamma A}}x} 
	=  \kappa' \norm{x-\J_{\gamma A}x}  
	\leq  \kappa'  \gamma \dist \lr{0, Ax}.
	\end{align*}

	\cref{lemma:metricallysubregularEQ:EQ}: This follows immediately from \cref{lemma:metricallysubregularEQ:A} and \cref{lemma:metricallysubregularEQ:IdJA} above. 
\end{proof}

%%%%%%%%%%%%%%%%%%%%%%%%%%%%%%%%%%%%%%%%%%%%%%%%%%%%
%%%%\section{Convergence of the perturbed or approximate method}%%%%%%%%%%%
%%%%%%%%%%%%%%%%%%%%%%%%%%%%%%%%%%%%%%%%%%%%%%%%%%%%
\section{Convergence of the Perturbed or Approximate Method} \label{section:PerturbedApproximate}

In this section, we borrow the terminology used in \cite{Lemaire1996}. Consider some problems with data $D$. 
Let $(\forall k \in \mathbb{N})$ $G_{k} : \mathcal{H} \to \mathcal{H}$ be the iteration mapping  that is defined from the data $D$.
The \emph{basic method}   is generated by conforming to the iteration scheme:
\begin{align*}
x_{0}=x \in \mathcal{H}  \quad \text{and} \quad (\forall k \in \mathbb{N}) ~ x_{k+1} =G_{k}x_{k}.
\end{align*}
With such a basic method, we have the following associated methods. 
\begin{enumerate}
	\item The  \emph{perturbed method}   is given by the iteration scheme:
	\begin{align*}
	x_{0}=x \in \mathcal{H} \quad \text{and} \quad (\forall k \in \mathbb{N}) ~ x_{k+1} =F_{k}x_{k},
	\end{align*}
	where for every  $k \in \mathbb{N}$, the iteration mapping $F_{k} : \mathcal{H} \to \mathcal{H}$ is closely related to $G_{k}$ and depends on  the perturbed (or approximate) data of $D$.
	\item  The \emph{translated basic  method}  is defined by the iteration scheme:
	\begin{align*}
	(\forall i \in \mathbb{N}) \quad \xi_{0}(i) = x_{i}  \in \mathcal{H} \quad \text{and} \quad (\forall k \in \mathbb{N}) ~ \xi_{k+1}(i) =G_{k+i}\xi_{k}(i).
	\end{align*}
	\item The \emph{approximate method}  is defined by the iteration scheme
	\begin{align*}
	x_{0}=x \in \mathcal{H} \quad \text{and} \quad  (\forall k \in \mathbb{N}) ~x_{k+1} =G_{k}x_{k} +e_{k},
	\end{align*}
	where $(e_{k})_{k \in \mathbb{N}}$ is the sequence of \emph{error terms}. 
\end{enumerate}
Note that, as stated in \cite[Remark~2.1]{Lemaire1996}, if  $\xi_{0}(i) = x$ and  $(\forall k \in \mathbb{N})$ $G_{k} \equiv G$, that is, the iterations mappings are independent on the iteration numbers $k$, then the translated basic method coincides with  the basic one.

\cref{prop:PkConverge} is essentially a special case of \cite[Lemma~2.1]{Lemaire1996}. As stated in \cite{Lemaire1996},    \cite[Lemma~2.1]{Lemaire1996}  is \enquote{substantially proved in \cite[Remark~14]{BrezisLions1978}}.
%\cref{prop:PkConverge} is essentially a special case of \cite[Lemma~2.1]{Lemaire1996} that is \enquote{substantially proved in \cite[Remark~14]{BrezisLions1978}}.
 Notice that \cite[Lemma~2.1]{Lemaire1996} is on a general topology weaker than the norm topology in   Banach spaces.  In \cref{prop:PkConverge}, we simplify the context and add  some  details on the proof of the weak and norm topologies in  Hilbert spaces to make the proof easier to understand, but the idea of the proof is almost the same as that of \cite[Lemma~2.1]{Lemaire1996}.

\begin{lemma} \label{prop:PkConverge}
	Let $\lr{\forall k \in \mathbb{N}}$ $G_{k} : \mathcal{H} \to \mathcal{H}$ be nonexpansive 
and let $(e_{k})_{k \in \mathbb{N}}$ be in $\mathcal{H}$.	Define 
	\begin{subequations} \label{eq:prop:PkConverge}
		\begin{align}
		&(\forall k \in \mathbb{N}) \quad y_{k+1}=G_{k}y_{k}+e_{k} \text{ and } y_{0} \in \mathcal{H}; \label{eq:prop:PkConverge:y}\\
		& (\forall i \in \mathbb{N}) (\forall k \in \mathbb{N}) \quad \xi_{k+1}(i)=G_{k+i}\xi_{k}(i) \text{ and } \xi_{0}(i)=y_{i}. \label{eq:prop:PkConverge:xi}
		\end{align}
	\end{subequations}
	Suppose that $\sum_{k \in \mathbb{N}} \norm{e_{k}} < \infty$ and that for every $i \in \mathbb{N}$, $(\xi_{k}(i))_{k \in \mathbb{N}}$ weakly converges to a point $\xi(i) \in \mathcal{H}$. Then the following statements hold. 
	\begin{enumerate}
		\item \label{prop:PkConverge:xiconverge} There exists a point $\bar{\xi} \in \mathcal{H}$ such that $(\xi(i))_{i \in \mathbb{N}}$ strongly converges to $\bar{\xi}$.
		\item \label{prop:PkConverge:ykweakly} $(y_{k})_{k\in \mathbb{N}}$ converges weakly to $\bar{\xi} = \lim_{i \to \infty} \xi(i)$.
		\item \label{prop:PkConverge:ykstrongly} Suppose that for every $i \in \mathbb{N}$, $(\xi_{k}(i))_{k \in \mathbb{N}}$ strongly converges to a point $\xi(i) \in \mathcal{H}$. Then $(y_{k})_{k\in \mathbb{N}}$ converges strongly to $\bar{\xi} = \lim_{i \to \infty} \xi(i)$.
	\end{enumerate}
\end{lemma}

\begin{proof}
	First, we claim that
	\begin{align}  \label{eq:prop:PkConverge:xie}
	(\forall k \in \mathbb{N}) (\forall i \in \mathbb{N}\smallsetminus\{0\}) \quad \norm{\xi_{k}(i) -\xi_{k+1}(i-1)} \leq  \norm{e_{i-1}}.
	\end{align} 
 In fact, for every $k \in \mathbb{N}$ and every $i \in \mathbb{N}\smallsetminus\{0\}$, if $k=0$, then 
 \begin{subequations}  \label{eq:prop:PkConverge:k=0}
 	\begin{align}
 	\norm{\xi_{0}(i) -\xi_{1}(i-1) } & \stackrel{\cref{eq:prop:PkConverge:xi}}{=} \norm{ y_{i}-G_{0+i-1}\xi_{0}(i-1) }\\
 	& \stackrel{\cref{eq:prop:PkConverge:xi}}{=}  \norm{ y_{i}-G_{i-1}y_{i-1}}\\ &\stackrel{\cref{eq:prop:PkConverge:y}}{=} \norm{e_{i-1}};
 	\end{align}
 \end{subequations}
%	\begin{align*}
%	\norm{\xi_{0}(i) -\xi_{1}(i-1) } \stackrel{\cref{eq:prop:PkConverge:xi}}{=} \norm{ y_{i}-G_{0+i-1}\xi_{0}(i-1) } \stackrel{\cref{eq:prop:PkConverge}}{=} \norm{ y_{i}-G_{i-1}y_{i-1}} \stackrel{\cref{eq:prop:PkConverge:y}}{=} \norm{e_{i-1}};
%	\end{align*}
	otherwise, by virtue of the nonexpansiveness of $G_{k+i-1}$  in the first   inequality below, we observe that
	\begin{align*}
	\norm{\xi_{k}(i) -\xi_{k+1}(i-1)}& \stackrel{\cref{eq:prop:PkConverge:xi}}{=} \norm{G_{k-1+i}\xi_{k-1}(i) -G_{k+i-1}\xi_{k}(i-1)  }\\
	&~\leq~~
	\norm{\xi_{k-1}(i)- \xi_{k}(i-1)  }\\
	&~\leq~~ \norm{\xi_{0}(i)-\xi_{1}(i-1)} \quad (\text{by induction})\\
	&\stackrel{\cref{eq:prop:PkConverge:k=0}}{=} \norm{e_{i-1}}.
%	&\stackrel{\cref{eq:prop:PkConverge:xi}}{=} \norm{\xi_{0}(i)-G_{0+i-1}\xi_{0}(i-1)}\\
%	&\stackrel{\cref{eq:prop:PkConverge:xi}}{=} \norm{y_{i}-G_{i-1}y_{i-1}} \\
%	&\stackrel{\cref{eq:prop:PkConverge:y}}{\leq} \norm{e_{i-1}}.
	\end{align*}
	Hence, we establish \cref{eq:prop:PkConverge:xie}.
	
	In view of \cref{eq:prop:PkConverge:xie}, for every $k \in \mathbb{N}$ and every $i \in \mathbb{N}$ with $i -k-1 \geq 0$,
	\begin{subequations}\label{eq:prop:PkConverge:yixi}
		\begin{align}
		\norm{y_{i} -\xi_{k+1}(i-k-1)} & \stackrel{\cref{eq:prop:PkConverge:xi}}{=} \norm{\xi_{0}(i) -\xi_{k+1}(i-k-1)}\\
		&~=~~\norm{\sum^{k}_{t=0} \lr{\xi_{t}(i-t) -\xi_{t+1}(i-t-1)}} \\
		&~\leq~~ \sum^{k}_{t=0}  \norm{\xi_{t}(i-t) -\xi_{t+1}(i-t-1) }\\
		&\stackrel{\cref{eq:prop:PkConverge:xie}}{\leq}~ \sum^{k}_{t=0}  \norm{e_{i-t-1}} =\sum^{i-1}_{j=i-k-1}  \norm{e_{j}}.
		\end{align}
	\end{subequations}
	For every $i \in \mathbb{N} \smallsetminus \{0\}$ and every $k \in \mathbb{N}$, substitute $i$ in \cref{eq:prop:PkConverge:yixi} by $i+k$ to derive that 
	\begin{align} \label{eq:prop:PkConverge:yikxiej}
	\norm{y_{i+k} -\xi_{k+1} (i-1)} \leq \sum^{i+k-1}_{j=i-1} \norm{e_{j}}.
	\end{align}

	\cref{prop:PkConverge:xiconverge}: As a consequence of \cref{eq:prop:PkConverge:xie}, for all $i $,  $p $, and    $k $ in $\mathbb{N}$  with $k \geq p$,
	\begin{subequations} \label{eq:prop:PkConverge:xisume}
		\begin{align}
		\norm{\xi_{k-p}(i+p) -\xi_{k}(i)} &~=~ \norm{\sum^{p}_{t=1} \left( \xi_{k-t} (i+t) -\xi_{k-t+1}(i+t-1) \right)}\\
		&~\leq~ \sum^{p}_{t=1}\norm{\xi_{k-t} (i+t) -\xi_{k-t+1}(i+t-1) }\\
		&\stackrel{\cref{eq:prop:PkConverge:xie}}{\leq}\sum^{p}_{t=1}\norm{e_{i+t-1}} =\sum^{i+p-1}_{j=i} \norm{e_{j}}.
		\end{align}
	\end{subequations}
	According to the assumption that for every $i \in \mathbb{N}$, $(\xi_{k}(i))_{k \in \mathbb{N}}$ weakly converges to a point $\xi(i) \in \mathcal{H}$, we observe that  
	\begin{align} \label{eq:prop:PkConverge:xiweakly}
(\forall i \in \mathbb{N})	(\forall p \in \mathbb{N}) \quad \xi_{k-p}(i+p) -\xi_{k}(i) \weakly \xi(i+p) -\xi(i) \text{ as } k \to \infty.
	\end{align}
Combine   \cite[Lemma~2.42]{BC2017} with  \cref{eq:prop:PkConverge:xiweakly} and \cref{eq:prop:PkConverge:xisume} to deduce that 
	\begin{align}\label{eq:prop:PkConverge:liminfleq}
(\forall i \in \mathbb{N}) (\forall p \in \mathbb{N}) \quad	\norm{ \xi(i+p) -\xi(i)} \stackrel{\cref{eq:prop:PkConverge:xiweakly} }{\leq} \liminf_{k \to \infty} \norm{ \xi_{k-p}(i+p) -\xi_{k}(i) } \stackrel{\cref{eq:prop:PkConverge:xisume}}{\leq}  \sum^{i+p-1}_{j=i} \norm{e_{j}}.
	\end{align}
	Let $\epsilon >0$. Because $\sum_{i \in \mathbb{N}}\norm{e_{i}} <\infty$, we know that there exists $I_{0} \in \mathbb{N}$ such that 
	\begin{align} \label{eq:prop:PkConverge:sumei}
	\sum_{i \geq I_{0}}\norm{e_{i}} < \epsilon.
	\end{align}
	Combine \cref{eq:prop:PkConverge:liminfleq} and \cref{eq:prop:PkConverge:sumei} to obtain that 
	\begin{align*}
	(\forall i \geq I_{0}) (\forall p \in \mathbb{N}) \quad \norm{ \xi(i+p) -\xi(i)} \leq \sum^{i+p-1}_{j=i} \norm{e_{j}} \leq \sum_{i \geq I_{0}}\norm{e_{i}} < \epsilon,
	\end{align*}
	which implies that $(\xi(i))_{i \in \mathbb{N}}$ is a Cauchy sequence in the Hilbert space $\mathcal{H}$. Hence, there exists a point $\bar{\xi} \in \mathcal{H}$ such that $(\xi(i))_{i \in \mathbb{N}}$ strongly converges to $\bar{\xi}$.

	\cref{prop:PkConverge:ykweakly}: Let $u \in \mathcal{H}$. Notice that for every  $i \in \mathbb{N} \smallsetminus \{0\}$ and $k \in \mathbb{N}$,
		\begin{align} \label{eq:prop:PkConverge:ykweakly:threeparts}
		\innp{y_{i+k} -\bar{\xi},u}
		= \innp{ y_{i+k} -\xi_{k+1}(i-1),u} +\innp{ \xi_{k+1}(i-1) -\xi(i-1) ,u}+\innp{\xi(i-1) -\bar{\xi},u}.
		\end{align}	
	Let $\varepsilon >0$. Assume that $u \neq 0$. 
	Due to \cref{prop:PkConverge:xiconverge} above, there exists $I_{1} \in \mathbb{N}$ such that 
	\begin{align} \label{eq:prop:PkConverge:ykweakly:xi}
	(\forall i \geq I_{1} ) \quad \innp{\xi(i-1) -\bar{\xi},u} < \frac{\varepsilon}{3}.
	\end{align}
	Since $\sum_{k \in \mathbb{N}} \norm{e_{k}} < \infty$, we see that there exists $I_{2} \in \mathbb{N}$ such that $\sum_{j\geq I_{2} -1 } \norm{e_{j}} < \frac{\varepsilon}{3 \norm{u}}$. Combine this with \cref{eq:prop:PkConverge:yikxiej} to get that every all $i \geq I_{2}  $ and for every $k \in \mathbb{N}$,
	\begin{align} \label{eq:prop:PkConverge:ykweakly:ykxi}
	\innp{ y_{i+k} -\xi_{k+1}(i-1),u}  \leq   \norm{y_{i+k} -\xi_{k+1} (i-1)} \norm{u} \leq \sum^{i+k-1}_{j=i-1} \norm{e_{j}} \norm{u} < \frac{\varepsilon}{3 }.
	\end{align}
	Set $I:=\max\{ I_{1}, I_{2}\}  $. Inasmuch as for every $i \in \mathbb{N}$, $(\xi_{k}(i))_{k \in \mathbb{N}}$ weakly converges to a point $\xi(i) \in \mathcal{H}$, there exists $K \in \mathbb{N}$ such that 
	\begin{align}\label{eq:prop:PkConverge:ykweakly:xiI}
	(\forall k \geq K) \quad \innp{ \xi_{k+1}(I-1) -\xi(I-1) ,u} < \frac{\varepsilon}{3 }.
	\end{align}
	Taking \cref{eq:prop:PkConverge:ykweakly:threeparts}, \cref{eq:prop:PkConverge:ykweakly:xi}, \cref{eq:prop:PkConverge:ykweakly:ykxi}, and \cref{eq:prop:PkConverge:ykweakly:xiI}  into account, we conclude that for every $n \geq I+K$,
	\begin{align}\label{eq:prop:PkConverge:ykweakly:weakly}
	\innp{y_{n}-\bar{\xi},u} &= \innp{y_{I+(n-I)} -\bar{\xi},u} < \varepsilon.
	\end{align}
	If $u=0$, then \cref{eq:prop:PkConverge:ykweakly:weakly} holds trivially. Altogether, $y_{n} \weakly \bar{\xi}$.
	
	\cref{prop:PkConverge:ykstrongly}:	
	Let $i \in \mathbb{N} \smallsetminus \{0\}$. Clearly, for every $k \in \mathbb{N}$,
		\begin{align} \label{eq:prop:PkConverge:ykstrongly:threeparts}
		 y_{i+k} -\bar{\xi} 
		= \left( y_{i+k} -\xi_{k+1}(i-1) \right) +\left( \xi_{k+1}(i-1) -\xi(i-1) \right) +\left( \xi(i-1) -\bar{\xi}\right).
		\end{align}
	Let $\varepsilon >0$. 
	Due to \cref{prop:PkConverge:xiconverge}, there exists $J_{1} \in \mathbb{N}$ such that 
	\begin{align}\label{eq:prop:PkConverge:ykstrongly:xi}
	(\forall i \geq J_{1} ) \quad \norm{\xi(i-1) -\bar{\xi}} < \frac{\varepsilon}{3}.
	\end{align}
	Based on \cref{eq:prop:PkConverge:yikxiej} and $\sum_{k \in \mathbb{N}} \norm{e_{k}} < \infty$, there exists $J_{2} \in \mathbb{N}$ such that
	\begin{align} \label{eq:prop:PkConverge:ykstrongly:yxi}
(\forall i \geq J_{2}) (\forall k \in \mathbb{N}) \quad	\norm{y_{i+k} -\xi_{k+1} (i-1)} \leq \sum^{i+k-1}_{j=i-1} \norm{e_{j}}  \leq \sum_{j \geq J_{2} -1} \norm{e_{j}} < \frac{\varepsilon}{3}.
	\end{align}
	Denote by $J:=\max \{ J_{1}, J_{2}\} $. Inasmuch as for every $i \in \mathbb{N}$, $(\xi_{k}(i))_{k \in \mathbb{N}}$ strongly converges  to $\xi(i)  $, there exists $N \in \mathbb{N} $ such that  
	\begin{align} \label{eq:prop:PkConverge:ykstrongly:xiJ}
	(\forall k \geq N) \quad \norm{ \xi_{k+1}(J-1) -\xi(J-1) } < \frac{\varepsilon}{3}.
	\end{align} 
	Bearing \cref{eq:prop:PkConverge:ykstrongly:threeparts}, \cref{eq:prop:PkConverge:ykstrongly:xi}, \cref{eq:prop:PkConverge:ykstrongly:yxi}, and \cref{eq:prop:PkConverge:ykstrongly:xiJ} in mind, we obtain that 
	\begin{align*}
	(\forall n \geq J+N) \quad \norm{ y_{n} -\bar{\xi}  } =\norm{y_{J+(n-J)} -\bar{\xi} } <\varepsilon,
	\end{align*}
	which means that $y_{n} \to \bar{\xi}$.
\end{proof}

\begin{theorem} \label{theorem:GkWeakStrongConvergence}
	Let $(\forall k \in \mathbb{N})$ $G_{k} : \mathcal{H} \to \mathcal{H}$ be nonexpansive and let  $(\forall k \in \mathbb{N})$ $F_{k} : \mathcal{H} \to \mathcal{H}$. 
	Let $(e_{k})_{k \in \mathbb{N}}$ be in $\mathcal{H}$ such that $\sum_{k \in \mathbb{N}} \norm{e_{k}} < \infty$.	Define 
	\begin{subequations}
		\begin{align}
		&(\forall k \in \mathbb{N}) \quad x_{k+1}=F_{k}x_{k}+e_{k} \text{ and } x_{0} \in \mathcal{H}; \label{eq:theorem:GkWeakStrongConvergence:y}\\
		&(\forall k \in \mathbb{N}) (\forall i \in \mathbb{N}) \quad \xi_{k+1}(i)=G_{k+i}\xi_{k}(i) \text{ and } \xi_{0}(i)=x_{i}. \label{eq:theorem:GkWeakStrongConvergence:xi}
		\end{align}
	\end{subequations}
	Then the  following statements hold. 
	\begin{enumerate}
		\item \label{theorem:GkWeakStrongConvergence:bounded} Suppose that $\cap_{k\in \mathbb{N}}\Fix G_{k} \neq \varnothing$. Let $\bar{x} \in \cap_{k\in \mathbb{N}}\Fix G_{k} $ and let  $(\forall k \in \mathbb{N})$ $\gamma_{k} \in \mathbb{R}_{+}$ with $\sum_{k \in \mathbb{N}} \gamma_{k} < \infty$.  Suppose that     $\sum_{k \in \mathbb{N}}  \norm{F_{k}\bar{x} -G_{k} \bar{x}} < \infty$ and that $(\forall k \in \mathbb{N})$ $F_{k}$
		is  $(1+\gamma_{k})$-Lipschitz continuous, i.e.,
		\begin{align}\label{eq:theorem:GkWeakStrongConvergence:Lipschitz}
		(\forall k \in \mathbb{N})(\forall  x\in \mathcal{H}) (\forall y \in \mathcal{H}) \quad \norm{F_{k}x- F_{k}y} \leq (1+\gamma_{k}) \norm{x-y}.
		\end{align}
		Then   $(\norm{x_{k} -\bar{x}})_{k \in \mathbb{N}}$ converges to a point in $\mathbb{R}_{+}$ and $(x_{k})_{k \in \mathbb{N}}$ is bounded. 
		
		\item \label{theorem:GkWeakStrongConvergence:xkconverges} Suppose that $\sum_{k \in \mathbb{N}} \norm{F_{k}x_{k} -G_{k}x_{k}} <\infty$ and that 
		for every $ i \in \mathbb{N}$,
		$(\xi_{k}(i))_{k \in \mathbb{N}}$ weakly   converges to a point $\xi(i) \in \mathcal{H}$. Then the following hold. 
		\begin{enumerate}
			\item There exists a point $\bar{\xi} \in \mathcal{H}$ such that $(\xi(i))_{i \in \mathbb{N}}$ strongly converges to $\bar{\xi}$.
			\item $(x_{k})_{k\in \mathbb{N}}$ weakly  converges to $\bar{\xi} = \lim_{i \to \infty} \xi(i)$. 
			\item Suppose further that $(\forall i \in \mathbb{N})$ $(\xi_{k}(i))_{k \in \mathbb{N}}$   converges strongly to a point $\xi(i) \in \mathcal{H}$. Then  $(x_{k})_{k\in \mathbb{N}}$   converges strongly to $\bar{\xi} = \lim_{i \to \infty} \xi(i)$. 
		\end{enumerate}
		
	\end{enumerate}
\end{theorem}

\begin{proof}
	\cref{theorem:GkWeakStrongConvergence:bounded}:	 Employing $(\forall k \in \mathbb{N})$ $\bar{x} =G_{k}\bar{x}$ in the first inequality below, we observe that for every $ k \in \mathbb{N}$,
	\begin{subequations}\label{theorem:GkWeakStrongConvergence:bounded:xk}
		\begin{align}
		\norm{x_{k+1}-\bar{x}} & \stackrel{\cref{eq:theorem:GkWeakStrongConvergence:y}}{=} \norm{ F_{k}x_{k}+e_{k} -\bar{x}}\\
		&~~\leq~~ \norm{F_{k}x_{k}-F_{k}\bar{x} } +\norm{F_{k}\bar{x} -G_{k}\bar{x} }+\norm{e_{k}}\\
		&\stackrel{\cref{eq:theorem:GkWeakStrongConvergence:Lipschitz}}{\leq} (1+\gamma_{k})\norm{ x_{k}-\bar{x} } +\norm{F_{k}\bar{x} -G_{k}\bar{x} }+\norm{e_{k}}.
		\end{align}
	\end{subequations}
	Combine the assumptions, $\sum_{k \in \mathbb{N}} \gamma_{k} < \infty$ and $\sum_{k \in \mathbb{N}} \left( \norm{F_{k}\bar{x} -G_{k}\bar{x} }+ \norm{e_{k}} \right)< \infty$, with 	\cref{fact:alphakINEQ} and \cref{theorem:GkWeakStrongConvergence:bounded:xk} to obtain the convergence of $(\norm{x_{k} -\bar{x}})_{k \in \mathbb{N}}$, which forces the boundedness of $(x_{k})_{k \in \mathbb{N}}$.

	\cref{theorem:GkWeakStrongConvergence:xkconverges}: Set $(\forall k \in \mathbb{N})$ $\tilde{e}_{k} :=F_{k}x_{k} -G_{k}x_{k} +e_{k}$.  Then \cref{eq:theorem:GkWeakStrongConvergence:y} becomes $x_{0} \in \mathcal{H}$ and
	\begin{align} \label{eq:theorem:GkWeakStrongConvergence:xk}
	(\forall k \in \mathbb{N}) \quad x_{k+1}=G_{k}x_{k} +F_{k}x_{k} -G_{k}x_{k}+e_{k} = G_{k}x_{k} + \tilde{e}_{k}.  
	\end{align}	
	By assumptions,
		\begin{align} \label{eq:theorem:GkWeakStrongConvergence:xi:boundtildeek}
		\sum_{k \in \mathbb{N}}	\norm{ \tilde{e}_{k} }  = 	\sum_{k \in \mathbb{N}} \norm{F_{k}x_{k} -G_{k}x_{k} +e_{k}}  
		  \leq \sum_{k \in \mathbb{N}} \norm{F_{k}x_{k} -G_{k}x_{k} } +  \sum_{k \in \mathbb{N}} \norm{e_{k}} <\infty.
		\end{align}
	Hence, bearing \cref{eq:theorem:GkWeakStrongConvergence:xk} and \cref{eq:theorem:GkWeakStrongConvergence:xi:boundtildeek} in mind and applying \cref{prop:PkConverge} with $(\forall k \in \mathbb{N})$ $y_{k}=x_{k}$ and $e_{k} =\tilde{e}_{k}$, we obtain all required results in \cref{theorem:GkWeakStrongConvergence:xkconverges}.
\end{proof}

\begin{remark}  \label{remark:GkWeakStrongConvergence}
	\begin{enumerate}
		\item 
		\cite[Proposition~2.1]{Lemaire1996}  proved that under a topology weaker than the norm topology in Banach spaces,  the convergence of translated basic methods ensures  the convergence of the perturbed method. 
		
		\label{remark:GkWeakStrongConvergence:I} \cref{theorem:GkWeakStrongConvergence} specifies the context of \cite[Proposition~2.1]{Lemaire1996} to the weak and norm topologies in Hilbert spaces  and affirm that  the weak convergence (resp.\,strong convergence) of translated basic methods necessitates the weak convergence (resp.\,strong convergence) of the associated method with both approximation and perturbation. In addition, in \cref{theorem:GkWeakStrongConvergence}\cref{theorem:GkWeakStrongConvergence:xkconverges}, we get rid of extra assumptions in \cite[Proposition~2.1]{Lemaire1996} and keep only necessary ones for deducing the required convergence, which will facilitate future applications. 
		
		\item 	The main idea of the proof of \cref{theorem:GkWeakStrongConvergence}\cref{theorem:GkWeakStrongConvergence:xkconverges} is essentially from the proof of  \cite[Proposition~2.1]{Lemaire1996}. Note that this idea 
		is also used in the proof of \cite[Theorem~4]{LiangFadiliPeyre2016} to prove the weak convergence of non-stationary Krasnosel'ski\v{\i}-Mann Iterations.
	\end{enumerate}
\end{remark}

In fact, 
if we consider \cite[Proposition~2.1]{Lemaire1996} only for weak and norm topologies in Hilbert spaces, then it reduces to \cref{cor:Fkxkconverge} below.
\begin{corollary} \label{cor:Fkxkconverge}
	Let $(\forall k \in \mathbb{N})$ $\gamma_{k} \in \mathbb{R}_{+}$ with $\sum_{k \in \mathbb{N}} \gamma_{k} < \infty$.	Let $(\forall k \in \mathbb{N})$ $G_{k} : \mathcal{H} \to \mathcal{H}$ be nonexpansive and let  $(\forall k \in \mathbb{N})$ $F_{k} : \mathcal{H} \to \mathcal{H}$ be $(1+\gamma_{k})$-Lipschitz continuous.	 	Define 
	\begin{subequations}
		\begin{align}
		&(\forall k \in \mathbb{N}) \quad x_{k+1}=F_{k}x_{k}  \text{ and } x_{0} \in \mathcal{H};\\ 
		& (\forall i \in \mathbb{N}) (\forall k \in \mathbb{N}) \quad \xi_{k+1}(i)=G_{k+i}\xi_{k}(i) \text{ and } \xi_{0}(i)=x_{i}.  
		\end{align}
	\end{subequations}
	Suppose that $\cap_{k \in \mathbb{N}} \Fix G_{k} \neq \varnothing$,  that    
	\begin{align}  \label{eq:cor:Fkxkconverge}
	(\forall \rho >0)\quad \sum_{k \in \mathbb{N}} \sup_{\norm{x} \leq \rho} \norm{F_{k}x-G_{k}x} <\infty,
	\end{align}
	and that $(\forall i \in \mathbb{N})$ $(\xi_{k}(i))_{k \in \mathbb{N}}$  converges weakly  (resp.\,strongly) to a point $\xi(i) \in \mathcal{H}$. Then  $(x_{k})_{k\in \mathbb{N}}$ converges weakly  (resp.\,strongly).
\end{corollary}

\begin{proof}
Let $\bar{x} \in \cap_{k \in \mathbb{N}} \Fix G_{k}$. 	It is clear that 
	\begin{align*}
	\sum_{k \in \mathbb{N}}  \norm{F_{k}\bar{x} -G_{k} \bar{x}} \leq \sum_{k \in \mathbb{N}} \sup_{\norm{x} \leq \norm{\bar{x}}} \norm{F_{k}x-G_{k}x} \stackrel{\cref{eq:cor:Fkxkconverge}}{<} \infty. 
	\end{align*}
	Hence, combine  assumptions above with  \cref{theorem:GkWeakStrongConvergence}\cref{theorem:GkWeakStrongConvergence:bounded} to deduce that there exists $\bar{\rho} \in \mathbb{R}_{++}$ such that  $(\forall k \in \mathbb{N})$ $\norm{x_{k}} \leq \bar{\rho} $.  Furthermore, we observe that 
	\begin{align*}
	\sum_{k \in \mathbb{N}} \norm{F_{k}x_{k} -G_{k}x_{k}} \leq  \sum_{k \in \mathbb{N}} \sup_{\norm{x} \leq \bar{\rho}} \norm{F_{k}x-G_{k}x}  \stackrel{\cref{eq:cor:Fkxkconverge}}{<} \infty.
	\end{align*}
	Therefore, due to \cref{theorem:GkWeakStrongConvergence}\cref{theorem:GkWeakStrongConvergence:xkconverges}, we obtain the required   convergence of $(x_{k})_{k\in \mathbb{N}}$. 
\end{proof}

To end  this section, we show the following easy but powerful result.  In particular, \cref{corollary:stabilityxkconverge} illustrates that the weak convergence (resp.\,strong convergence) of the translated basic methods implies  the weak convergence (resp.\,strong convergence) of the associated approximate method. 
\begin{corollary} \label{corollary:stabilityxkconverge}
	Let   $(\forall k \in \mathbb{N})$ $G_{k} : \mathcal{H} \to \mathcal{H}$ be nonexpansive and
	let $(e_{k})_{k \in \mathbb{N}}$ be in $\mathcal{H}$ such that $\sum_{k \in \mathbb{N}} \norm{e_{k}} < \infty$.	Define 
	\begin{align*}
	&(\forall k \in \mathbb{N}) \quad x_{k+1}=G_{k}x_{k}+e_{k} \text{ and } x_{0} \in \mathcal{H};\\
	&(\forall k \in \mathbb{N}) (\forall i \in \mathbb{N}) \quad \xi_{k+1}(i)=G_{k+i}\xi_{k}(i) \text{ and } \xi_{0}(i)=x_{i}.  
	\end{align*}
  Let $C$ be a nonempty closed subset of $\mathcal{H}$. 
	Suppose that $(\forall i \in \mathbb{N})$ $\lr{\xi_{k}(i) }_{k \in \mathbb{N}}$	converges  weakly  (resp.\,converges strongly) to a point in $C$. Then $(x_{k})_{k \in \mathbb{N}}$  converges weakly  (resp.\,converges strongly) to a point in $C$.
\end{corollary}

\begin{proof}
	It is trivial that $\sum_{k \in \mathbb{N}} \norm{G_{k}x_{k} -G_{k}x_{k}} =0 <\infty$.
Therefore, it is not difficult to derive the required results by applying \cref{theorem:GkWeakStrongConvergence}\cref{theorem:GkWeakStrongConvergence:xkconverges} with $(\forall k \in \mathbb{N})$ $F_{k} =G_{k}$. 

Alternatively, we can also get the required result directly from \cref{prop:PkConverge}.
\end{proof}

%%%%%%%%%%%%%%%%%%%%%%%%%%%%%%%%%%%%%%%%%%%%%%%%%%%%
%%\section{Convergence of Relaxation Variants of  Krasnosel'ski\v{\i}-Mann iterations}%%%%%%
%%%%%%%%%%%%%%%%%%%%%%%%%%%%%%%%%%%%%%%%%%%%%%%%%%%%
 \section{Convergence of Relaxation Variants of Krasnosel'ski\v{\i}-Mann Iterations} \label{sec:ConvergenceKMI}

\cref{fact:KMIterations} shows the weak convergence of the Krasnosel'ski\v{\i}-Mann iterations. Note that although in \cite[Proposition~5.16]{BC2017}, $\alpha $ is in $ \left]0,1\right[$, based on \cite[Theorem~5.15]{BC2017} and the proof of  \cite[Proposition~5.16]{BC2017}, there is little difficulty to extend the range of $\alpha $ to $\left]0,1\right]$.

\begin{fact} {\rm \cite[Proposition~5.16]{BC2017}} \label{fact:KMIterations}
	Let $\alpha \in \left]0,1\right]$, let $T: \mathcal{H} \to \mathcal{H}$ be an $\alpha$-averaged operator such that $\Fix T \neq \varnothing$, let $\lr{\lambda_{k}}_{k \in \mathbb{N}}$ be a sequence in $\left[0,\frac{1}{\alpha}\right]$ such that $\sum_{k \in \mathbb{N}} \lambda_{k} \lr{1 - \alpha \lambda_{k}} =\infty$, and let $x_{0} \in \mathcal{H}$. Set 
	\begin{align*}
	(\forall k \in \mathbb{N}) \quad x_{k+1} = x_{k} + \lambda_{k} \lr{Tx_{k} -x_{k}}.
	\end{align*}
	Then $(x_{k})_{k \in \mathbb{N}}$ converges weakly to a point in $\Fix T$.
\end{fact}

Denote by $\ell^{1}_{+} := \left\{ (\alpha_{k})_{k \in \mathbb{N}}   ~:~ \sum_{k \in \mathbb{N}} \alpha_{k} < \infty \text{ and } (\forall k \in \mathbb{N}) \alpha_{k} \in \mathbb{R}_{+} \right\}$.
\begin{definition}  \label{defin:FMonotone}  {\rm \cite[Definitions~5.1 and 5.32]{BC2017}} 
	Let $C$ be a nonempty subset of $\mathcal{H}$ and let $(x_{k})_{k \in \mathbb{N}}$ be a sequence in $\mathcal{H}$. Then 
	\begin{enumerate}
		\item \label{defin:FMonotone:Fejer} $(x_{k})_{k \in \mathbb{N}}$ is \emph{Fej\'er  monotone with respect to $C$} 	\index{Fej\'er  monotone} if  
		\begin{align*}
		(\forall x \in C) (\forall k \in \mathbb{N}) \quad \norm{x_{k+1} -x} \leq \norm{x_{k} -x}.	
		\end{align*}
		\item  \label{defin:FMonotone:QuasiFejer}  $(x_{k})_{k \in \mathbb{N}}$ is \emph{quasi-Fej\'er  monotone with respect to $C$} if
		\begin{align*}
		(\forall x \in C) \lr{\exists (\varepsilon_{k})_{k \in \mathbb{N}} \in \ell^{1}_{+}} (\forall k \in \mathbb{N})  \quad \norm{x_{k+1} -x }^{2} \leq \norm{x_{k} -x}^{2} +\varepsilon_{k}.
		\end{align*}
	\end{enumerate}
\end{definition}

\subsection{Inexact non-stationary Krasnosel'ski\v{\i}-Mann iterations }\label{section:NonstationaryKMI}

In this subsection, we provide some properties of the inexact  non-stationary Krasnosel'ski\v{\i}-Mann iterations.

\cref{theorem:KMBasic}   is inspired by \cite[Theorem~5.15 and Proposition~5.34]{BC2017}. In particular, \cref{theorem:KMBasic}\cref{theorem:KMBasic:quasiFejer} generalizes \cite[Proposition~5.34(i)]{BC2017} by replacing the nonexpansive operator $T$ therein with a sequence of averaged operators $(T_{k})_{k \in \mathbb{N}}$; moreover, \cref{theorem:KMBasic}\cref{theorem:KMBasic:liminf} extends  \cite[Theorem~5.15(ii)]{BC2017} from the exact Krasnosel'ski\v{\i}-Mann iterations to inexact  non-stationary Krasnosel'ski\v{\i}-Mann iterations.
\begin{theorem} \label{theorem:KMBasic}
	Let $(\alpha_{k})_{k\in \mathbb{N}}$ be in $\left]0,1\right]$, let $(\forall k \in \mathbb{N})$ $\lambda_{k} \in \left[ 0,\frac{1}{\alpha_{k}}  \right]$,  and let $(\forall k \in \mathbb{N})$  $T_{k} : \mathcal{H} \to \mathcal{H}$ be $\alpha_{k}$-averaged. Let $x_{0} $ be in $ \mathcal{H}$. Define  
	\begin{align*}   
	\lr{\forall k \in \mathbb{N}} \quad x_{k+1} =(1-\lambda_{k})x_{k} +\lambda_{k} T_{k}x_{k} +\eta_{k}e_{k}.
	\end{align*}  
	Suppose that $ \cap_{k\in \mathbb{N}}\Fix T_{k} \neq \varnothing$. Then the following statements hold. 
	\begin{enumerate}
		
		\item \label{theorem:KMBasic:induction} $(\forall \bar{x} \in \cap_{k\in \mathbb{N}}\Fix T_{k})$ $(\forall k \in \mathbb{N})$ $\norm{x_{k+1} -\bar{x}} \leq  \norm{x_{0} -\bar{x}} + \sum^{k}_{i =0} \eta_{i}\norm{e_{i}}$.
		
		\item  \label{theorem:KMBasic:sumek} Suppose that $\sum_{k \in \mathbb{N}}   \eta_{k} \norm{e_{k}} <\infty$. Then the following hold. 
		\begin{enumerate}
		\item  \label{theorem:KMBasic:quasiFejer} $(x_{k})_{k \in \mathbb{N}}$ is quasi-Fej\'er monotone with respect to $\cap_{k\in \mathbb{N}}\Fix T_{k}$.
		
		\item  \label{theorem:KMBasic:sum} $\sum_{k \in \mathbb{N}} \lambda_{k}  \lr{\frac{1}{\alpha_{k}} -\lambda_{k}} \norm{x_{k} -T_{k}x_{k}}^{2} < \infty$.
		
		\item \label{theorem:KMBasic:liminf} Suppose that $\sum_{k \in \mathbb{N}} \lambda_{k}  \lr{\frac{1}{\alpha_{k}} -\lambda_{k}} = \infty$. Then $\liminf_{k \to \infty} \norm{x_{k} -T_{k}x_{k}}=0$.
		\item  \label{theorem:KMBasic:lim} Suppose that $\sum_{k \in \mathbb{N}} \lambda_{k}  \lr{\frac{1}{\alpha_{k}} -\lambda_{k}} = \infty$ and that $\lim_{k \to \infty} \norm{x_{k} -T_{k}x_{k}}$ exists	$($e.g., $\lr{ \norm{x_{k} -T_{k}x_{k}} }_{k \in \mathbb{N}}$ is decreasing$)$. 
		Then  $	\lim_{k \to \infty} \norm{x_{k} -T_{k}x_{k}}=0$.
	
		\item  \label{theorem:KMBasic:l2} Suppose that $\liminf_{k \to \infty}  \lambda_{k}  \lr{\frac{1}{\alpha_{k}} -\lambda_{k}} >0$ $($e.g., $\liminf_{k \to \infty}  \lambda_{k} >0$ and $\limsup_{k \to \infty} \lambda_{k} < \frac{1}{\limsup_{k \to \infty} \alpha_{k}} <\infty$$)$. Then $\sum_{k \in \mathbb{N}}   \norm{x_{k} -T_{k}x_{k}}^{2} < \infty$. Consequently,  $\lim_{k \to \infty} \norm{x_{k} -T_{k}x_{k}}=0$.
		 
\end{enumerate}		
	\end{enumerate}
\end{theorem}

\begin{proof}
	Let $\bar{x} \in \cap_{k\in \mathbb{N}}\Fix T_{k} $. Set 
	\begin{align}  \label{eq:theorem:KMBasic:varepsilon}
	(\forall k \in \mathbb{N}) \quad y_{k} := (1-\lambda_{k})x_{k} +\lambda_{k} T_{k}x_{k} \quad \text{and} \quad  \varepsilon_{k}:=\eta_{k} \norm{e_{k}} \lr{2\norm{y_{k} -\bar{x}} +\eta_{k} \norm{e_{k}}}.
	\end{align}
	For every $k \in \mathbb{N}$, apply \cref{lemma:yxzx}\cref{lemma:yxzx:norm}$\&$\cref{lemma:yxzx:lambdaalpha} with $x=x_{k}$, $y_{x}=y_{k}$, $z_{x} =x_{k+1}$, $T=T_{k}$, $\alpha =\alpha_{k}$, $\lambda =\lambda_{k}$, $\eta =\eta_{k}$, and $e=e_{k}$ to derive that 
	\begin{subequations}
		\begin{align}
		&\norm{y_{k} -\bar{x}}^{2} \leq \norm{x_{k} -\bar{x}}^{2} -\lambda_{k} \lr{\frac{1}{\alpha_{k}} -\lambda_{k}} \norm{x_{k} -T_{k}x_{k}}^{2};\label{eq:theorem:KMBasic:yk}\\
		&\norm{x_{k+1} -\bar{x}}^{2} \leq \norm{x_{k} -\bar{x}}^{2} -\lambda_{k}  \lr{\frac{1}{\alpha_{k}} -\lambda_{k}} \norm{x_{k} -T_{k}x_{k}}^{2} +\varepsilon_{k};\label{eq:theorem:KMBasic:xk+1}\\
		&\norm{x_{k+1} -\bar{x}} \leq \norm{y_{k} -\bar{x}} +\eta_{k} \norm{e_{k} } \leq \norm{x_{k} -\bar{x}} +\eta_{k}  \norm{e_{k}}\label{eq:theorem:KMBasicxkyk}.
		\end{align}
	\end{subequations}
	
	\cref{theorem:KMBasic:induction}: This is clear from \cref{eq:theorem:KMBasicxkyk} by induction.
	
	\cref{theorem:KMBasic:sumek}: Bearing \cref{eq:theorem:KMBasicxkyk} and $\sum_{k \in \mathbb{N}}  \eta_{k}  \norm{ e_{k}} <\infty$ in mind and  applying  \cref{fact:alphakINEQ} with $(\forall k \in \mathbb{N})$
	$\alpha_{k} = \norm{x_{k} -\bar{x}} $, $\beta_{k} =\gamma_{k} \equiv 0$, and $\varepsilon_{k}=\eta_{k} \norm{e_{k} }$, 	
	we know that  $  \lim_{k \to \infty} \norm{x_{k} -\bar{x}} $ exists in $ \mathbb{R}_{+}$. This together with the assumption that $\sum_{k \in \mathbb{N}} \eta_{k}  \norm{ e_{k}} <\infty$ entails that 
	\begin{align} \label{eq:theorem:KMBasic:M1M2}
	M_{1}:= \sup_{k \in \mathbb{N}} \norm{x_{k} -\bar{x}}< \infty \quad \text{and} \quad M_{2} := \sup_{k \in \mathbb{N}} \eta_{k} \norm{  e_{k}} < \infty.
	\end{align}
	Hence, 
	\begin{align*}
	\sum_{k \in \mathbb{N}} \varepsilon_{k} \stackrel{\cref{eq:theorem:KMBasic:varepsilon}}{=} &  \sum_{k \in \mathbb{N}}   \eta_{k} \norm{e_{k}} \lr{2\norm{y_{k} -\bar{x}} +\eta_{k} \norm{e_{k}}} \\
	\stackrel{\cref{eq:theorem:KMBasic:yk}}{\leq} &  \sum_{k \in \mathbb{N}}  \eta_{k} \norm{e_{k}} \lr{2\norm{x_{k} -\bar{x}} +\eta_{k} \norm{e_{k}}} \\
	\stackrel{\cref{eq:theorem:KMBasic:M1M2}}{\leq} & \lr{2M_{1}+M_{2}} \sum_{k \in \mathbb{N}}   \eta_{k} \norm{e_{k}} < \infty,
	\end{align*}
	which, combining with \cref{eq:theorem:KMBasic:xk+1}, \cref{defin:FMonotone}\cref{defin:FMonotone:QuasiFejer}, and \cref{fact:alphakINEQ}, guarantees the  results in \cref{theorem:KMBasic:quasiFejer} and \cref{theorem:KMBasic:sum} hold.
	
	\cref{theorem:KMBasic:liminf} and \cref{theorem:KMBasic:l2} are immediate from  \cref{theorem:KMBasic:sum} above.
	Moreover,	\cref{theorem:KMBasic:lim} is clear from \cref{theorem:KMBasic:liminf} above.
	
Altogether, the proof is complete.
\end{proof}

\begin{corollary}\label{prop:KMBasic:imply}
	Let $(\alpha_{k})_{k\in \mathbb{N}}$ be in $\left]0,1\right]$, let $(\forall k \in \mathbb{N})$ $\lambda_{k} \in \left[ 0,\frac{1}{\alpha_{k}}  \right]$,  and let $(\forall k \in \mathbb{N})$  $T_{k} : \mathcal{H} \to \mathcal{H}$ be $\alpha_{k}$-averaged. Let $x_{0} $ be in $ \mathcal{H}$. Define  
	\begin{align*}   
	\lr{\forall k \in \mathbb{N}} \quad x_{k+1} =(1-\lambda_{k})x_{k} +\lambda_{k} T_{k}x_{k} +\eta_{k}e_{k}.
	\end{align*}  
	Suppose that $C:= \cap_{k\in \mathbb{N}}\Fix T_{k} \neq \varnothing$ and that $\sum_{k \in \mathbb{N}}  \eta_{k} \norm{  e_{k}} <\infty$.
	Then the following statements hold.
	\begin{enumerate}
		\item \label{prop:KMBasic:imply:normconverge}  For every $\bar{x} \in C$, $(\norm{x_{k} -\bar{x}})_{k \in \mathbb{N}}$ converges. 
		\item  \label{prop:KMBasic:imply:bounded} $(x_{k})_{k \in \mathbb{N}}$ is bounded.
		\item  \label{prop:KMBasic:imply:dist} $( \dist (x_{k}, C))_{k \in \mathbb{N}}$ converges. 
		\item  \label{prop:KMBasic:imply:weaklyconver}	$(x_{k})_{k \in \mathbb{N}}$ converges weakly to a point in $C$ if and only if $\Omega \lr{  (x_{k})_{k \in \mathbb{N}} } \subseteq C$.
	\end{enumerate} 
\end{corollary}

\begin{proof}
	The required results follow directly from \cite[Theorem~5.33]{BC2017}  and  \cref{theorem:KMBasic}\cref{theorem:KMBasic:quasiFejer}.	
\end{proof}

\subsection{Convergence of inexact Krasnosel'ski\v{\i}-Mann iterations}
In \cref{proposition:KMIterations} below, we invoke \cref{corollary:stabilityxkconverge} and extend the weak convergence of  the Krasnosel'ski\v{\i}-Mann iterations from its  exact version to its inexact version. 
\begin{proposition} \label{proposition:KMIterations}
	Let $\alpha \in \left]0,1\right]$, let $T: \mathcal{H} \to \mathcal{H}$ be an $\alpha$-averaged operator such that $\Fix T \neq \varnothing$, and let $\lr{\lambda_{k}}_{k \in \mathbb{N}}$ be a sequence in $\left]0,\frac{1}{\alpha}\right]$ such that $\sum_{k \in \mathbb{N}} \lambda_{k} \lr{1 - \alpha \lambda_{k}} =\infty$. Let $(e_{k})_{k \in \mathbb{N}}$ be in $\mathcal{H}$ and let $(\eta_{k})_{k \in \mathbb{N}}$ be in $\mathbb{R}_{+}$ with $\sum_{k \in \mathbb{N}} \eta_{k} \norm{e_{k}} < \infty$. Let $x_{0} \in \mathcal{H}$. Set
	\begin{align} \label{eq:proposition:KMIterations:xk}
	(\forall k \in \mathbb{N}) \quad x_{k+1} = x_{k} + \lambda_{k} \lr{Tx_{k} -x_{k}} +\eta_{k} e_{k}.
	\end{align}
	Then $(x_{k})_{k \in \mathbb{N}}$ converges weakly to a point in $\Fix T$.
\end{proposition}

\begin{proof}
	Define 
	\begin{align*}
	(\forall k \in \mathbb{N}) \quad G_{k}:= \lr{1-\lambda_{k}} \Id +\lambda_{k} T.
	\end{align*}
	Then \cref{eq:proposition:KMIterations:xk} becomes
	\begin{align} \label{eq:proposition:KMIterations:xkG}
	(\forall k \in \mathbb{N}) \quad x_{k+1}=G_{k}x_{k}+\eta_{k} e_{k}. 
	\end{align}   
	Because   $T :\mathcal{H} \to \mathcal{H}$ is $\alpha $-averaged  and $(\forall k \in \mathbb{N})$ $\lambda_{k} \in \left]0,\frac{1}{\alpha }\right]$, we know, by \cref{fact:Averagedlambdaalpha},  that $(\forall k \in \mathbb{N})$ $G_{k}$ is $\lambda_{k} \alpha$-averaged, which implies that $(\forall k \in \mathbb{N})$ $G_{k}$ is nonexpansive. 
	Because  $(\forall k \in \mathbb{N})$ $\lambda_{k} \neq 0$, it is easy to see that $(\forall k \in \mathbb{N})$ $\Fix G_{k} = \Fix T$,
	which, combined with \cref{fact:quasinonexCC},  entails that $\cap_{k \in \mathbb{N}} \Fix G_{k} =\Fix T $ is nonempty and closed.
	
	Notice that 
	\begin{align*}
	(\forall i \in \mathbb{N}) \quad  \sum_{k \in \mathbb{N}} \lambda_{k} \lr{1 - \alpha \lambda_{k}} =\infty \Leftrightarrow  \sum_{k \in \mathbb{N}} \lambda_{k+i} \lr{1 - \alpha \lambda_{k+i}} =\infty.
	\end{align*}
	Set 
	\begin{align}\label{eq:proposition:KMIterations:xkxi}
 (\forall i \in \mathbb{N})	(\forall k \in \mathbb{N})  \quad \xi_{k+1}(i)=G_{k+i}\xi_{k}(i) \text{ and } \xi_{0}(i)=x_{i}.  
	\end{align}
	Hence, for every $i \in \mathbb{N}$, applying  \cref{fact:KMIterations} with $(\forall k \in \mathbb{N})$ $x_{k}=\xi_{k}(i)$ and $\lambda_{k} =\lambda_{k+i}$,
	we have that  
	\begin{align} \label{eq:proposition:KMIterations:weakly}
	(\forall i \in \mathbb{N})  \quad  \xi_{k}(i) \weakly \xi(i) \in \Fix T.
	\end{align}

	Therefore, bearing \cref{eq:proposition:KMIterations:xkG}, \cref{eq:proposition:KMIterations:xkxi}, and \cref{eq:proposition:KMIterations:weakly}  in mind and applying
	\cref{corollary:stabilityxkconverge} with $C=\Fix T$ and  $(\forall k \in \mathbb{N})$ $e_{k} =\eta_{k} e_{k}$, we obtain the required weak convergence of the sequence $(x_{k})_{k \in \mathbb{N}}$.
\end{proof}

\begin{remark}
	\cref{proposition:KMIterations} reduces to \cite[Proposition~5.34(iii)]{BC2017} if $\alpha=1$ and $(\forall k \in \mathbb{N})$ $\eta_{k} \equiv \lambda_{k}$ although the proof of 	\cref{proposition:KMIterations}  has nothing to do with that of  \cite[Proposition~5.34]{BC2017}. 
\end{remark}

\subsection{Convergence of inexact non-stationary  Krasnosel'ski\v{\i}-Mann iterations} \label{section:NonstationaryKMIConverge}

In this subsection, we show the weak and strong convergence of the inexact version of the non-stationary Krasnosel'ski\v{\i}-Mann iterations.

In the following result, we deduce the weak convergence of the inexact   non-stationary Krasnosel'ski\v{\i}-Mann iterations from the weak convergence of the exact   Krasnosel'ski\v{\i}-Mann iterations.
\begin{theorem} \label{theorem:TkxikConver}
	Let $\alpha \in \left]0,1\right]$ and $(\forall k \in \mathbb{N})$ $\lambda_{k} \in \left]0,\frac{1}{\alpha} \right]$, let $T: \mathcal{H} \to \mathcal{H}$ be $\alpha$-averaged with $\Fix T \neq \varnothing$, and let $(\forall k \in \mathbb{N})$ $T_{k} : \mathcal{H} \to \mathcal{H}$.
	Let $(e_{k})_{k \in \mathbb{N}}$ be a sequence in $\mathcal{H}$ and  let $(\eta_{k})_{k \in \mathbb{N}}$ be in $\mathbb{R}_{+}$ such that $\sum_{k \in \mathbb{N}} \eta_{k}\norm{e_{k}} < \infty$. Let $x_{0}$ be in $\mathcal{H}$. Define
	\begin{align}  \label{eq:theorem:TkxikConver:xk}
	(\forall k \in \mathbb{N}) \quad x_{k +1 }=  (1-\lambda_{k}) x_{k} + \lambda_{k}T_{k}   x_{k}+ \eta_{k} e_{k}.
	\end{align}

	Suppose that $\sum_{k \in \mathbb{N}} \lambda_{k} \norm{T_{k}x_{k}-Tx_{k}} <\infty$.  Then the following statements hold.
	
	\begin{enumerate}
		\item \label{theorem:TkxikConver:general}  Define 
		\begin{align} \label{eq:theorem:TkxikConver:xik}
		(\forall i \in \mathbb{N}) (\forall k \in \mathbb{N}) \quad	 \xi_{k+1}(i)=(1-\lambda_{k+i})  \xi_{k}(i)  + \lambda_{k+i}T  \xi_{k}(i) \text{ and } \xi_{0}(i)=x_{i}.   
		\end{align}
		Suppose that for every $ i \in \mathbb{N}$,
		$(\xi_{k}(i))_{k \in \mathbb{N}}$ weakly   converges to a point $\xi(i) \in \mathcal{H}$. Then the following hold. 
		\begin{enumerate}
			\item  \label{theorem:TkxikConver:general:xi}  There exists a point $\bar{\xi} \in \mathcal{H}$ such that $(\xi(i))_{i \in \mathbb{N}}$ strongly converges to $\bar{\xi}$.
			\item \label{theorem:TkxikConver:general:weak}  $(x_{k})_{k\in \mathbb{N}}$ weakly  converges to $\bar{\xi} = \lim_{i \to \infty} \xi(i)$. 
			\item \label{theorem:TkxikConver:general:strong}  Suppose further that $(\forall i \in \mathbb{N})$ $(\xi_{k}(i))_{k \in \mathbb{N}}$   converges strongly to a point $\xi(i) \in \mathcal{H}$. Then  $(x_{k})_{k\in \mathbb{N}}$   converges strongly to $\bar{\xi} = \lim_{i \to \infty} \xi(i)$. 
		\end{enumerate}
		\item \label{theorem:TkxikConver:Fk} Suppose that $\sum_{k \in \mathbb{N}} \lambda_{k}  \lr{\frac{1}{\alpha} -\lambda_{k}} = \infty$. 
		Then  $(x_{k})_{k\in \mathbb{N}}$ weakly    converges to a point $\bar{x}$ in $\Fix T$.
	\end{enumerate}
\end{theorem}

\begin{proof}
		\cref{theorem:TkxikConver:general}: Define 
		\begin{align}  \label{eq:theorem:TkxikConver}
		(\forall k \in \mathbb{N}) \quad   F_{k}:= (1-\lambda_{k}) \Id + \lambda_{k}T_{k} \text{ and }  G_{k}:= (1-\lambda_{k}) \Id + \lambda_{k}T.
		\end{align}  
		
		Inasmuch as $T: \mathcal{H} \to \mathcal{H}$ is $\alpha$-averaged and $(\forall k \in \mathbb{N})$ $\lambda_{k} \in \left]0,\frac{1}{\alpha} \right]$, via \cref{fact:Averagedlambdaalpha},  we  know that 	 $(\forall k \in \mathbb{N})$ $G_{k}$ is nonexpansive. 
	According to the construction of the operators $(G_{k})_{k \in \mathcal{H}}$ and $(F_{k})_{k \in \mathcal{H}}$ in \cref{eq:theorem:TkxikConver},  we observe that  \cref{eq:theorem:TkxikConver:xk}  and \cref{eq:theorem:TkxikConver:xik} become, respectively, 
	\begin{align*}
	&(\forall k \in \mathbb{N}) \quad x_{k+1}=F_{k}x_{k}+\eta_{k}e_{k};\\
	& (\forall i \in \mathbb{N}) (\forall k \in \mathbb{N}) \quad \xi_{k+1}(i)=G_{k+i}\xi_{k}(i) \text{ and } \xi_{0}(i)=x_{i}.    
	\end{align*}
	Moreover, applying  \cref{eq:theorem:TkxikConver} again, we see that
	\begin{align*}
	\sum_{k \in \mathbb{N}} \norm{F_{k}x_{k} -G_{k}x_{k}} <\infty \Leftrightarrow 	\sum_{k \in \mathbb{N}} \lambda_{k} \norm{T_{k}x_{k}-Tx_{k}} <\infty.
	\end{align*}
	Hence,  \cref{theorem:TkxikConver:general}  is clear from    \cref{theorem:GkWeakStrongConvergence}\cref{theorem:GkWeakStrongConvergence:xkconverges}.
	
	\cref{theorem:TkxikConver:Fk}: Note that  
	\begin{align*}
(\forall i \in \mathbb{N}) \quad 	\sum_{k \in \mathbb{N}} \lambda_{k}  \lr{\frac{1}{\alpha} -\lambda_{k}} = \infty  \Leftrightarrow 	\sum_{k \in \mathbb{N}} \lambda_{k+i}  \lr{\frac{1}{\alpha} -\lambda_{k+i}} = \infty.
	\end{align*}
	 For every $i \in \mathbb{N}$, applying   \cref{fact:KMIterations} with $(\forall k \in \mathbb{N})$ $x_{k}=\xi_{k}(i)$ and $\lambda_{k} =\lambda_{k+i}$,
	we establish that   $(\xi_{k}(i))_{k \in \mathbb{N}}$ weakly   converges to a point $\xi(i) \in \Fix T$. 
	
Furthermore, bearing \cref{fact:quasinonexCC} and \cref{theorem:TkxikConver:general:xi} in mind, we know that the limit $\bar{\xi}$ of $(\xi(i))_{i \in \mathbb{N}}$ must be in $\Fix T$ as well. Therefore, we obtain the required weak convergence of  $(x_{k})_{k \in \mathbb{N}}$ by \cref{theorem:TkxikConver:general:weak} above.
\end{proof}

\begin{remark} \label{remark:TkxikConver:Fk}
	\cref{theorem:TkxikConver}\cref{theorem:TkxikConver:Fk} 
	generalizes \cite[Theorem~4]{LiangFadiliPeyre2016} from the following aspects.
	\begin{itemize}
		\item The nonexpansive operator $T$ in \cite[Theorem~4]{LiangFadiliPeyre2016} is extended to   the $\alpha$-averaged operator $T$ with $\alpha \in \left]0,1\right]$ in 	\cref{theorem:TkxikConver}\cref{theorem:TkxikConver:Fk}  (notice that $1$-averaged operator is nonexpansive).
		\item The assumption $\inf_{k \in \mathbb{N}} \lambda_{k} \lr{1-\lambda_{k}} > 0 $ in \cite[Theorem~4]{LiangFadiliPeyre2016} is replaced by a more general assumption $\sum_{k \in \mathbb{N}} \lambda_{k}  \lr{\frac{1}{\alpha} -\lambda_{k}} = \infty$.
		\item The corresponding assumptions that  
		\begin{align*}
		(\forall \rho \in \mathbb{R}_{+}) \quad \sum_{k \in \mathbb{N}} \sup_{\norm{x} \leq \rho} \lambda_{k} \norm{T_{k}x-Tx} < \infty
		\end{align*}
		and that    $(\forall k \in \mathbb{N})$  $F_{k}$ is $(1+\gamma_{k})$-Lipschitz continuous, where $( \gamma_{k} )_{k \in \mathbb{N}} $ is in $\mathbb{R}_{+}$ such that $\sum_{k \in \mathbb{N}} \gamma_{k} < \infty$,
		  in \cite[Theorem~4]{LiangFadiliPeyre2016}  are simplified as   $	\sum_{k \in \mathbb{N}} \lambda_{k+1} \norm{T_{k+1}x_{k}-Tx_{k}} <\infty$
		in 	\cref{theorem:TkxikConver}\cref{theorem:TkxikConver:Fk}. 
	\end{itemize}
\end{remark}

To show a strong convergence of the inexact   non-stationary Krasnosel'ski\v{\i}-Mann iterations in \cref{proposition:KMI} below, we first deduce the following result on the exact non-stationary Krasnosel'ski\v{\i}-Mann iterations. 
\begin{theorem} \label{theorem:KMI}
	Let $(\forall k \in \mathbb{N})$ $T_{k} :\mathcal{H} \to \mathcal{H}$ be $\alpha_{k}$-averaged with $\alpha_{k} \in \left]0,1\right]$ and $C := \cap_{k \in \mathbb{N}} \Fix T_{k}  \neq \varnothing$.  Let $\bar{x} \in \cap_{k\in \mathbb{N}} \Fix T_{k}$, let $(\forall k \in \mathbb{N})$ $\lambda_{k} \in \left]0,\frac{1}{\alpha_{k}}\right[\,$, and let $x_{0}$ be in $\mathcal{H}$. Define
	\begin{align*}
	\lr{\forall k \in \mathbb{N}} \quad x_{k+1} = \lr{1-\lambda_{k}}x_{k} +\lambda_{k} T_{k}x_{k}.
	\end{align*}
	Then the following statements hold. 
	\begin{enumerate}
		\item  \label{theorem:KMI:Fejer} $\lr{x_{k}}_{k \in \mathbb{N}}$ is Fej\'er monotone with respect to $C$.
		\item  \label{theorem:KMI:MSubR} Suppose that  $(\forall k \in \mathbb{N})$  $\Id - T_{k}$ is metrically subregular at $\bar{x}$ for $0 = \lr{\Id -T_{k}}\bar{x}$, i.e., for every $k \in \mathbb{N}$,
		\begin{align*}
		\lr{\exists \gamma_{k} >0} \lr{\exists \delta_{k} >0} \lr{\forall x \in B[\bar{x}; \delta_{k}]} \quad \dist \lr{x, \lr{\Id - T_{k}}^{-1} 0 } \leq \gamma_{k} \norm{x -T_{k}x}.
		\end{align*}
		Suppose that $\delta: =\inf_{k \in \mathbb{N}} \delta_{k} >0$ and    $x_{0} \in B[\bar{x};\delta]$.  Define $(\forall k \in \mathbb{N})$ $\rho_{k}:=\lr{  1- \frac{\lambda_{k} \lr{1-\lambda_{k} \alpha_{k}}}{ \alpha_{k}\gamma_{k}^{2}} }^{\frac{1}{2}}$. Then the following hold. 
		\begin{enumerate}
			\item \label{theorem:KMI:Distance} $\lr{\forall k \in \mathbb{N}}$ $x_{k} \in B[\bar{x};\delta]$ and $  \dist  \lr{x_{k+1}, \Fix T_{k}} \leq \rho_{k} \dist  \lr{x_{k}, \Fix T_{k} }$.
			\item \label{theorem:KMI:MSubR:moreassum} Suppose that $\lr{\forall k \in \mathbb{N}}$ $\Fix T_{k} =C$ and that  $0< \underline{\lambda}:=\liminf_{k \to \infty} \lambda_{k} \leq \overline{\lambda}:=\limsup_{k \to \infty} \lambda_{k} < \frac{1}{\alpha}$ where $\alpha:=\limsup_{k \to \infty} \alpha_{k} >0 $ and $\gamma :=\limsup_{k \to \infty}\gamma_{k}  < \infty$.
			Define $\rho:=\limsup_{k \to \infty} \rho_{k}$.
			Then the following hold. 
			\begin{enumerate}
				\item \label{theorem:KMI:MSubR:rho} $0\leq \rho \leq \lr{1 - \frac{\underline{\lambda} \lr{ \frac{1}{\alpha}-\overline{\lambda}}}{ \gamma^{2}}}^{\frac{1}{2}} <  \frac{1}{ \lr{  1 + \frac{\underline{\lambda} \lr{ \frac{1}{\alpha}-\overline{\lambda}}}{ \gamma^{2}} }^{\frac{1}{2}}  }  <1$.
		
			\item \label{theorem:KMI:MSubR:moreassum:dist} There exist  $K \in \mathbb{N}$ and $\mu \in \left]\rho, 1\right[$ such that for every $k \geq K$,
		$  \dist  \lr{x_{k+1}, C}  \leq \mu \dist  \lr{x_{k}, C}$.
		\item \label{theorem:KMI:Lindear} There exist $\hat{x} \in C$, $K \in \mathbb{N}$, and $\mu \in \left]\rho, 1\right[$  such that for every  $k \geq K$,
		\begin{align} \label{eq:theorem:KMI:EQ}
		(\forall k \geq K) \quad \norm{x_{k} -\hat{x}} \leq 2 \mu^{k-K} \dist \lr{x_{K}, C}.
		\end{align}
		Consequently, $\lr{x_{k}}_{k \in \mathbb{N}}$ converges $R$-linearly to a point $\hat{x} \in C$.
		In particular, if $C$ is affine, then $\hat{x} =\Pro_{C}x_{0}$.		
%				\item \label{theorem:KMI:MSubR:moreassum:dist} There exists $K \in \mathbb{N}$ such that  $(\forall k \geq K)$
%				$  \dist  \lr{x_{k+1}, C}  \leq \rho \dist  \lr{x_{k}, C}$.
%				\item \label{theorem:KMI:Lindear} There exist $\hat{x} \in C$ and  $K \in \mathbb{N}$ such that  
%				\begin{align} \label{eq:theorem:KMI:EQ}
%				(\forall k \geq K) \quad \norm{x_{k} -\hat{x}} \leq 2 \rho^{k-K} \dist \lr{x_{K}, C}.
%				\end{align}
%				Consequently, $\lr{x_{k}}_{k \in \mathbb{N}}$ converges $R$-linearly to a point $\hat{x} \in C$.
%				In particular, if $C$ is affine, then $\hat{x} =\Pro_{C}x_{0}$.
			\end{enumerate}
			
		\end{enumerate} 
	\end{enumerate} 
\end{theorem}

\begin{proof}
	\cref{theorem:KMI:Fejer}: According to \cref{defin:FMonotone}\cref{defin:FMonotone:Fejer},  applying \cref{theorem:KMBasic}\cref{theorem:KMBasic:induction} with $(\forall k \in \mathbb{N})$ $e_{k} \equiv 0$ and $\eta_{k}   \equiv 0$, we obtain the required Fej\'er monotonicity.

	\cref{theorem:KMI:Distance}: Using \cref{theorem:KMI:Fejer} above and 
	invoking the assumption, $x_{0} \in B[\bar{x};\delta]$, we establish by induction  that
	\begin{align*}
	(\forall k \in \mathbb{N}) \quad \norm{x_{k+1} -\bar{x}} \leq \norm{x_{k}-\bar{x}} \leq \cdots \leq \norm{x_{0} -\bar{x}} \leq \delta.
	\end{align*}
	Hence,	 for every $k \in \mathbb{N}$, the inequality $  \dist  \lr{x_{k+1}, \Fix T_{k} } \leq \rho_{k} \dist  \lr{x_{k}, \Fix T_{k} }$
	follows directly from \cref{proposition:TFMetricSubreg}\cref{proposition:TFMetricSubreg:B} with replacing 
	 $x=x_{k}$, $\alpha =\alpha_{k}$, $\lambda =\lambda_{k}$, $T=T_{k}$, $\Fix T=\Fix T_{k} $, and $F_{\lambda}x =x_{k+1}$.
	
	\cref{theorem:KMI:MSubR:moreassum}:
	It is easy to deduce  \cref{theorem:KMI:MSubR:moreassum}i.\,by our assumptions and \cref{proposition:TFMetricSubreg}\cref{proposition:TFMetricSubreg:rho}. 
	
	The result in \cref{theorem:KMI:MSubR:moreassum}ii.\,is immediate from  \cref{theorem:KMI:Distance}   above.
	
	As a consequence of \cref{fact:quasinonexCC}, we know that $C$ is nonempty closed and convex. 
	Taking \cref{theorem:KMI:Fejer}, \cref{theorem:KMI:MSubR:moreassum}ii.\,, and \cite[Theorem~5.12]{BC2017} into account, we easily obtain \cref{eq:theorem:KMI:EQ} which implies the desired $R$-linear convergence of $\lr{x_{k}}_{k \in \mathbb{N}}$.	
	
	In addition,  if $C$ is affine,  then, via the  convergence of $(x_{k})_{k \in \mathbb{N}}$ proved above and \cite[Propositin~5.9(ii)]{BC2017}, we obtain $\hat{x} =\Pro_{C}x_{0}$.
	
	Altogether, the proof is complete.
\end{proof}

\begin{remark} \label{Remark:KMI}
	\cref{theorem:KMI} generalizes \cite[Corollary~5.4]{BauschkeNollPhan2015} from the following two aspects.
	\begin{itemize}
		\item We extend \cite[Corollary~5.4]{BauschkeNollPhan2015} from boundedly linearly regularity to metrically subregularity.
		\item The  operator $T$ in \cite[Corollary~5.4]{BauschkeNollPhan2015} is   replaced by the affine combination $\lr{1-\lambda_{k}}\Id + \lambda_{k} T_{k}$ in \cref{theorem:KMI}  which reduces to $T$ when $(\forall k \in \mathbb{N})$ $\lambda_{k} \equiv 1$ and $T_{k} \equiv T$.
	\end{itemize}
\end{remark}

Applying \cref{corollary:stabilityxkconverge} and the convergence of the exact version of the non-stationary  Krasnosel'ski\v{\i}-Mann iterations presented in \cref{theorem:KMI}, we obtain the strong convergence of the inexact version of the non-stationary  Krasnosel'ski\v{\i}-Mann iterations in \cref{proposition:KMI} below.
\begin{theorem}  \label{proposition:KMI}
	Let $(\forall k \in \mathbb{N})$ $T_{k} :\mathcal{H} \to \mathcal{H}$ be $\alpha_{k}$-averaged with $\alpha_{k} \in \left]0,1\right]$ and $C :=  \Fix T_{k}  \neq \varnothing$.  Let $\bar{x} \in C$ and let $(\forall k \in \mathbb{N})$ $\lambda_{k} \in \left]0,\frac{1}{\alpha_{k}}\right[\,$. 
	Suppose that  $(\forall k \in \mathbb{N})$  $\Id - T_{k}$ is metrically subregular at $\bar{x}$ for $0 = \lr{\Id -T_{k}}\bar{x}$, i.e., for every $k \in \mathbb{N}$,
	\begin{align*}
	\lr{\exists \gamma_{k} >0} \lr{\exists \delta_{k} >0} \lr{\forall x \in B[\bar{x}; \delta_{k}]} \quad \dist \lr{x, \lr{\Id - T_{k}}^{-1} 0 } \leq \gamma_{k} \norm{x -T_{k}x}.
	\end{align*}
	Suppose  that 
	 $\delta: =\inf_{k \in \mathbb{N}} \delta_{k} >0$
	and that $\epsilon:= \sum_{k \in \mathbb{N}} \eta_{k} \norm{e_{k}} < \delta$.	 
	Suppose that $0< \underline{\lambda}:=\liminf_{k \to \infty} \lambda_{k} \leq \overline{\lambda}:=\limsup_{k \to \infty} \lambda_{k} < \frac{1}{\alpha}$ where $\alpha:=\limsup_{k \to \infty } \alpha_{k}  >0$ and $\gamma :=\limsup_{k \to \infty }\gamma_{k}  < \infty$.

	Let $\hat{\delta} \in \left]0, \delta -\epsilon\right]$ and let  $x_{0} \in B\left[\bar{x}; \hat{\delta}\right]$. 
	Define
	\begin{align}  \label{eq:proposition:KMI:xk}
	\lr{\forall k \in \mathbb{N}} \quad x_{k+1} = \lr{1-\lambda_{k}}x_{k} +\lambda_{k} T_{k}x_{k} + \eta_{k} e_{k}.
	\end{align}
	
		Then $\lr{x_{k}}_{k \in \mathbb{N}}$ converges  strongly  to a point in $  C$. In particular,   if $C$ is affine, then $\lr{x_{k}}_{k \in \mathbb{N}}$ converges strongly to $\Pro_{C}x_{0}$.
	
%	Then $\lr{x_{k}}_{k \in \mathbb{N}}$ converges    to a point in $  C$. In particular,   if $C$ is affine, then $\lr{x_{k}}_{k \in \mathbb{N}}$ converges to $\Pro_{\Fix T}x_{0}$.	
\end{theorem}

\begin{proof}
	Define 
		\begin{align} \label{eq:proposition:KMI:GKTK}
	(\forall k \in \mathbb{N}) \quad G_{k}:= \lr{1-\lambda_{k}} \Id +\lambda_{k} T_{k}.
	\end{align}
%	\begin{align*}
%	(\forall k \in \mathbb{N}) \quad G_{k}:= \lr{1-\lambda_{k}} \Id +\lambda_{k} T_{k}.
%	\end{align*}
	Then \cref{eq:proposition:KMI:xk} becomes
	\begin{align}  \label{eq:proposition:KMI:Gk}
	(\forall k \in \mathbb{N}) \quad x_{k+1}=G_{k}x_{k}+ \eta_{k}e_{k}. 
	\end{align}
	Inasmuch as $(\forall k \in \mathbb{N})$ $T_{k} :\mathcal{H} \to \mathcal{H}$ is $\alpha_{k}$-averaged  and $(\forall k \in \mathbb{N})$ $\lambda_{k} \in \left]0,\frac{1}{\alpha_{k}}\right[\,$, we know, by \cref{fact:Averagedlambdaalpha},  that $(\forall k \in \mathbb{N})$ $G_{k}$ is $\lambda_{k} \alpha_{k}$-averaged, which implies that $(\forall k \in \mathbb{N})$ $G_{k}$ is nonexpansive.
	
	Because $x_{0} \in B\left[\bar{x}; \hat{\delta}\right]$ and $\bar{x} \in C=\cap_{k\in \mathbb{N}} \Fix T_{k}$, due to  \cref{theorem:KMBasic}\cref{theorem:KMBasic:induction}, we observe that
	\begin{align} \label{eq:proposition:KMI:delta}
	(\forall k \in \mathbb{N}) \quad \norm{x_{k+1} -\bar{x}} \leq  \norm{x_{0} -\bar{x}} + \sum^{k}_{i =0} \eta_{i}\norm{e_{i}} \leq \hat{\delta} +\epsilon \leq \delta.
	\end{align}
	
	Define
	\begin{align}  \label{eq:proposition:KMI:xik}
	(\forall i \in \mathbb{N}) 
	(\forall k \in \mathbb{N})   \quad \xi_{k+1}(i)=G_{k+i} \xi_{k}(i)   \quad \text{and} \quad \xi_{0}(i)=x_{i}.
	% \stackrel{\cref{eq:proposition:KMI:delta}}{\in}  B \left[\bar{x};\delta\right].
	\end{align}
	In view of \cref{eq:proposition:KMI:delta}, we know that $(\forall i \in \mathbb{N})$  $\xi_{0}(i) \in B \left[\bar{x};\delta\right]$.
	For every $i \in \mathbb{N}$, applying  \cref{theorem:KMI}\cref{theorem:KMI:MSubR:moreassum}  
%	For every $i \in \mathbb{N}$, applying  \cref{theorem:KMI} 
	with 
	$(\forall k \in \mathbb{N})$ $x_{k}=\xi_{k}(i)$, $T_{k}=T_{k+i}$, and $\lambda_{k}=\lambda_{k+i}$, we obtain that
	the sequence   $\lr{\xi_{k}(i)}_{k \in \mathbb{N}}$
		converges strongly to  a point in $C$. 
		In addition, as a result of \cref{fact:quasinonexCC} and \cref{eq:proposition:KMI:GKTK}, $C =\cap_{k \in \mathbb{N}} \Fix G_{k}$ is nonempty and closed. 
%		Moreover, as a result of \cref{fact:quasinonexCC}, $C$ is closed. 
	
	Therefore, 
	invoking the strong convergence of $(\forall i \in \mathbb{N})$ $\lr{\xi_{k}(i)}_{k \in \mathbb{N}}$, taking \cref{eq:proposition:KMI:Gk} and  \cref{eq:proposition:KMI:xik} into account, and employing \cref{corollary:stabilityxkconverge}, we obtained that $x_{k} \to \hat{x} \in C$.
 If $C$ is affine,  then,  the  convergence of $(x_{k})_{k \in \mathbb{N}}$  and \cite[Propositin~5.9(ii)]{BC2017} lead to $\hat{x} =\Pro_{C}x_{0}$.
\end{proof}

%%%%%%%%%%%%%%%%%%%%%%%%%%%%%%%%%%%%%%%%%%%%%%%%%%%%%%%%
%%%%%%%%%%%Section{Convergence of Generalized Proximal Point Algorithms}%%%%%%%%%
%%%%%%%%%%%%%%%%%%%%%%%%%%%%%%%%%%%%%%%%%%%%%%%%%%%%%%%%

\section{Convergence of Generalized Proximal Point Algorithms} \label{section:GPPA}

Throughout this section, $A :\mathcal{H} \to 2^{\mathcal{H}}$ is maximally monotone with $\zer A \neq \varnothing$ and 
\begin{align} \label{eq:GPPA}
(\forall k \in \mathbb{N}) \quad x_{k+1} = \lr{1-\lambda_{k}}x_{k} +\lambda_{k} \J_{c_{k} A}x_{k} +\eta_{k}e_{k}, 
\end{align}
where $x_{0} \in \mathcal{H}$ is the \emph{initial point} and $(\forall k \in \mathbb{N})$ $\lambda_{k} \in \left[0,2\right]$ and $\eta_{k} \in \mathbb{R}_{+}$ are the \emph{relaxation coefficients},  $c_{k} \in \mathbb{R}_{++}$ is the \emph{regularization coefficient},  and $e_{k} \in \mathcal{H}$ is the \emph{error term}. 

Generalized proximal point algorithms generate the iteration sequence by conforming to the scheme  \cref{eq:GPPA}. The classic proximal point algorithm generates the iteration sequence by following \cref{eq:GPPA} with $(\forall k \in \mathbb{N})$ $\lambda_{k} \equiv 1$, $e_{k} \equiv 0$, and $\eta_{k} \equiv 0$.
In this section, we investigate the weak and strong convergence of generalized proximal point algorithms for solving the monotone inclusion problem, i.e., finding a point in $\zer A$.

\subsection{Iteration sequences generated by generalized proximal point algorithms}

In this subsection, we provide some properties of the iteration sequences generalized by \cref{eq:GPPA}.

\begin{lemma}  \label{lemma:GPPAPro} 
Let $\bar{x} \in \zer A$. Set
	\begin{align*}
	(\forall k \in \mathbb{N}) \quad y_{k}= \lr{1-\lambda_{k}}x_{k} +\lambda_{k} \J_{c_{k} A}x_{k}    \text{ and } \varepsilon_{k}  = \eta_{k} \norm{e_{k}} \lr{2\norm{y_{k} -\bar{x}} +\eta_{k} \norm{e_{k}}}.
	\end{align*}
	Then the following hold. 
	\begin{enumerate}
		\item \label{lemma:GPPAPro:yk} $(\forall k \in \mathbb{N})$  $\norm{y_{k} -\bar{x}}^{2} \leq \norm{x_{k} -\bar{x}}^{2} -\lambda_{k} \lr{2 -\lambda_{k}} \norm{x_{k} -\J_{c_{k} A}x_{k}}^{2}$.
		\item \label{lemma:GPPAPro:Ineq}  $(\forall k \in \mathbb{N})$  $\norm{x_{k+1} -\bar{x}}^{2} \leq \norm{x_{k} -\bar{x}}^{2} -\lambda_{k}  \lr{2 -\lambda_{k}} \norm{x_{k} -\J_{c_{k} A}x_{k}}^{2} +\varepsilon_{k}$.
	
		\item \label{lemma:GPPAPro:sumek} 	Suppose that $\sum_{k \in \mathbb{N}} \eta_{k} \norm{e_{k}} < \infty$. Then the following hold.
		\begin{enumerate}
			\item \label{lemma:GPPAPro:sumvarepsion} $\sum_{k \in \mathbb{N}} \varepsilon_{k} = \sum_{k \in \mathbb{N}}  \eta_{k} \norm{e_{k}} \lr{2\norm{y_{k} -\bar{x}} +\eta_{k} \norm{e_{k}}} < \infty $
				\item \label{lemma:GPPAPro:exists} $\lr{\forall \bar{x} \in \zer A}$ $\lim_{k \to \infty} \norm{x_{k} -\bar{x}}$ exists.
			\item \label{lemma:GPPAPro:sumlambda}  $\sum_{k \in \mathbb{N}} \lambda_{k}  \lr{2 -\lambda_{k}} \norm{x_{k} -\J_{c_{k} A}x_{k}}^{2} <\infty.$
\end{enumerate}

		 	\item  \label{lemma:GPPAPro:OR}  Suppose that  $\sum_{k \in \mathbb{N}} \lambda_{k} \lr{2 -\lambda_{k}} = \infty$. Suppose that $c:= \inf_{k \in \mathbb{N}} c_{k} >0$ or  that $\sum_{k \in \mathbb{N}} c_{k}^{2}=\infty$ and $(\forall k \in \mathbb{N})$ $\lambda_{k} \equiv 1$. Then $\sum_{k \in \mathbb{N}} c^{2}_{k} \lambda_{k} \lr{2 -\lambda_{k}} = \infty$.
		 	
		\item \label{lemma:GPPAPro:lim} 	Suppose that 
		$\sum_{k \in \mathbb{N}} \eta_{k} \norm{e_{k}} < \infty$ and  $\sum_{k \in \mathbb{N}} \lambda_{k} \lr{2 -\lambda_{k}} = \infty$. Suppose that $ \inf_{k \in \mathbb{N}} c_{k} >0 $ or  that $\sum_{k \in \mathbb{N}} c_{k}^{2}=\infty$ and $(\forall k \in \mathbb{N})$ $\lambda_{k} \equiv 1$. 
	Then $\liminf_{k \to \infty}  \norm{\frac{1}{c_{k}} \lr{x_{k} - \J_{c_{k} A}x_{k} }} =0$.
	\end{enumerate} 
\end{lemma}

\begin{proof}
\cref{lemma:GPPAPro:yk}$\&$\cref{lemma:GPPAPro:Ineq}: For every $k \in \mathbb{N}$, applying 	\cref{lemma:resolvents:yxzx} with $x =x_{k}$, $y_{x}=y_{k}$, $z_{x}=x_{k+1}$, $\lambda =\lambda_{k}$, $\gamma =c_{k}$, $\eta =\eta_{k}$, and $e=e_{k}$, we deduce the required inequalities in \cref{lemma:GPPAPro:yk} and \cref{lemma:GPPAPro:Ineq}.

\cref{lemma:GPPAPro:sumek}:  Because $\sum_{k \in \mathbb{N}} \eta_{k} \norm{e_{k}} < \infty$, due to
\cite[Proposition~3.3(ii)]{OuyangWeakStrongGPPA2021},
we know that  $(x_{k})_{k \in \mathbb{N}}$ is bounded. This connected with $\sum_{k \in \mathbb{N}} \eta_{k} \norm{e_{k}} < \infty$ and  \cref{lemma:GPPAPro:yk} above implies that $\lr{2\norm{y_{k} -\bar{x}} +\eta_{k} \norm{e_{k}}}_{k \in \mathbb{N}}$ is bounded.  Hence, it is easy to see that $\sum_{k \in \mathbb{N}} \varepsilon_{k} = 	\sum_{k \in \mathbb{N}}  \eta_{k} \norm{e_{k}} \lr{2\norm{y_{k} -\bar{x}} +\eta_{k} \norm{e_{k}}} < \infty$,
which, combined with \cref{fact:alphakINEQ}  and \cref{lemma:GPPAPro:Ineq} above, entails that 
$\lim_{k \to \infty} \norm{x_{k} -\bar{x}}$ exists and that $\sum_{k \in \mathbb{N}} \lambda_{k}  \lr{2 -\lambda_{k}} \norm{x_{k} -\J_{c_{k} A}x_{k}}^{2} <\infty$.

\cref{lemma:GPPAPro:OR}: If $c= \inf_{k \in \mathbb{N}} c_{k} >0$, then $\sum_{k \in \mathbb{N}} c^{2}_{k} \lambda_{k} \lr{2 -\lambda_{k}} \geq c^{2}  \sum_{k \in \mathbb{N}} \lambda_{k} \lr{2 -\lambda_{k}} = \infty$.
If $\sum_{k \in \mathbb{N}} c_{k}^{2}=\infty$ and $(\forall k \in \mathbb{N})$ $\lambda_{k} \equiv 1$, then $\sum_{k \in \mathbb{N}} c^{2}_{k} \lambda_{k} \lr{2 -\lambda_{k}} = \sum_{k \in \mathbb{N}} c^{2}_{k} = \infty$.

\cref{lemma:GPPAPro:lim}: Because  $\sum_{k \in \mathbb{N}} \eta_{k} \norm{e_{k}} < \infty$, via \cref{lemma:GPPAPro:sumlambda} above, we observe that 
\begin{align*}
\sum_{k \in \mathbb{N}} \lambda_{k} \lr{ 2 -\lambda_{k}} c_{k}^{2} \norm{\frac{1}{c_{k}} \lr{x_{k} - \J_{c_{k} A}x_{k} }}^{2} <\infty,
\end{align*}  
which, connected with  \cref{lemma:GPPAPro:OR}, yields that $\liminf_{k \to \infty}  \norm{\frac{1}{c_{k}} \lr{x_{k} - \J_{c_{k} A}x_{k} }} =0$.
\end{proof}

Note that the assumption $(\forall k \in \mathbb{N})$  $c_{k+1} \geq c_{k}$ in \cref{lemma:RkJckINEQ} below  also appears in the seminal work \cite[Theorem~2]{Rockafellar1976} although a weaker requirement works therein. 
\begin{lemma} \label{lemma:RkJckINEQ}
	Define $(\forall k \in \mathbb{N})$ $ \R_{k}:= 2 \J_{c_{k} A} -\Id$.	Suppose that 
	$\sum_{k \in \mathbb{N}} \eta_{k} \norm{e_{k}} < \infty$.
	Then the following statements hold.
	\begin{enumerate}
		\item \label{lemma:RkJckINEQ:JckA} $(\forall k \in \mathbb{N})$ $  \frac{1}{2} \norm{\R_{k+1} x_{k} -\R_{k} x_{k} } \leq \abs{1-\frac{c_{k+1}}{c_{k}} } \norm{x_{k} - \J_{c_{k} A}x_{k}} $.
		\item \label{lemma:RkJckINEQ:xk+1} $(\forall k \in \mathbb{N})$ $ \norm{x_{k+1} - \J_{c_{k+1} A}x_{k+1}} \leq \lr{ 1 + \abs{1-\frac{c_{k+1}}{c_{k}}}} \norm{ x_{k} - \J_{c_{k } A}x_{k } } + \eta_{k} \norm{e_{k}}$.
		\item  \label{lemma:RkJckINEQ:cknonincreasing} Assume that $(\forall k \in \mathbb{N})$  $c_{k+1} \geq c_{k}$.
		 Then 
		\begin{align*}
		\lr{\forall k \in \mathbb{N}} \quad \frac{1}{c_{k+1}}   \norm{x_{k+1} - \J_{c_{k+1} A}x_{k+1}} \leq  \frac{1}{c_{k}}  \norm{ x_{k} - \J_{c_{k } A}x_{k } } + \frac{ \eta_{k} }{ c_{k+1} } \norm{e_{k}}.
		\end{align*}
		\item  \label{lemma:RkJckINEQ:to0} Suppose that $\sum_{k \in \mathbb{N}} \lambda_{k} \lr{2 -\lambda_{k}} = \infty$  and  that $(\forall k \in \mathbb{N})$  $c_{k+1} \geq c_{k}$. Then $\lim_{k \to \infty}  \frac{1}{c_{k}} \lr{x_{k} - \J_{c_{k} A}x_{k} } =0$. 
	\end{enumerate}
\end{lemma}

\begin{proof}
Because $A$ is maximally monotone, 
via \cref{fact:cAMaximallymonotone} and
	\cite[Propositions~4.4 and 20.22]{BC2017}, 
	  we know that $(\forall k \in \mathbb{N})$ $ \R_{k}$  is nonexpansive. Clearly, for every $k \in \mathbb{N}$,
	\begin{subequations}
		\begin{align}
		&\R_{k} - \Id = 2 \lr{ \J_{c_{k} A} -\Id } \Leftrightarrow  \J_{c_{k} A} -\Id  = \frac{1}{2} \lr{ \R_{k} - \Id }; \label{eq:lemma:RkJckINEQ:Rk}\\
		&\R_{k+1} -\R_{k} = 2 \lr{ \J_{c_{k+1} A} -\J_{c_{k} A} }. \label{eq:lemma:RkJckINEQ:Rk+1} 
		\end{align}
	\end{subequations}

	Combine \cref{eq:GPPA} and \cref{eq:lemma:RkJckINEQ:Rk} to derive that 
	\begin{align*}
\lr{\forall k \in \mathbb{N}} \quad 	x_{k+1} =   x_{k} + \lambda_{k}  \lr{\J_{c_{k} A}x_{k} - x_{k}  }  +\eta_{k}e_{k} 
	= x_{k} + \frac{ \lambda_{k}   }{ 2 }\lr{\R_{k} x_{k} - x_{k}  }  +\eta_{k}e_{k},
	\end{align*}
	which follows immediately by
	\begin{align}\label{eq:lemma:RkJckINEQ:xkRk}
	(\forall k \in \mathbb{N}) \quad x_{k+1} - \R_{k}x_{k} = \lr{1 - \frac{\lambda_{k}}{2}} \lr{x_{k} - \R_{k}x_{k} } + \eta_{k} e_{k}.
	\end{align}
	
\cref{lemma:RkJckINEQ:JckA}:	 For every $ k \in \mathbb{N}$, applying \cref{corollary:JcA}\cref{corollary:JcA:mulambda} in the second equality and employing the nonexpansiveness of $\J_{c_{k+1} A}$ in the inequality below, we obtain that
	\begin{align*}
	 \frac{1}{2} \norm{\R_{k+1} x_{k} -\R_{k} x_{k} } 
	\stackrel{\cref{eq:lemma:RkJckINEQ:Rk+1}}{=}&  \norm{\J_{c_{k+1} A}x_{k} -\J_{c_{k} A}x_{k} }\\  
	~=~&  \norm{\J_{c_{k+1} A}x_{k} -\J_{c_{k+1} A} \lr{    \frac{c_{k+1}}{c_{k}} x_{k} + \lr{1 - \frac{c_{k+1}}{c_{k}}} \J_{c_{k} A}x_{k}} }\\  
	~ \leq~ &\abs{1-\frac{c_{k+1}}{c_{k}} } \norm{x_{k} - \J_{c_{k} A}x_{k}},
	\end{align*}  
	which follows directly by \cref{lemma:RkJckINEQ:JckA}.
	
	 \cref{lemma:RkJckINEQ:xk+1}:	 Now, for every $k \in \mathbb{N}$, by virtue of the nonexpansiveness of $\R_{k}$ in the second inequality below,  
	\begin{align*}
 \norm{x_{k+1} - \J_{c_{k+1} A}x_{k+1}}  
	\stackrel{\cref{eq:lemma:RkJckINEQ:Rk}}{=}& \frac{1}{2} \norm{ \R_{k+1} x_{k+1} -x_{k+1} }\\
	~\leq~& \frac{1}{2} \norm{\R_{k+1} x_{k+1} -\R_{k+1} x_{k} } + \frac{1}{2} \norm{\R_{k+1} x_{k} -\R_{k} x_{k} } + \frac{1}{2} \norm{\R_{k} x_{k} -x_{k+1} } \\
	~\leq~& \frac{1}{2} \norm{ x_{k+1} - x_{k} } + \frac{1}{2} \norm{\R_{k+1} x_{k} -\R_{k} x_{k} } + \frac{1}{2} \norm{\R_{k} x_{k} -x_{k+1} }\\
	\stackrel{\cref{eq:lemma:RkJckINEQ:xkRk}}{\leq}& \frac{1}{2} \norm{ x_{k+1} - x_{k} } + \frac{1}{2} \norm{\R_{k+1} x_{k} -\R_{k} x_{k} } + \frac{ 1 - \frac{\lambda_{k}}{2}}{2}  \norm{\R_{k} x_{k} -x_{k} } +\frac{1}{2}\eta_{k} \norm{e_{k}}\\
	~\stackrel{\text{\cref{lemma:RkJckINEQ:JckA}}}{\leq} ~& \frac{1}{2} \norm{ x_{k+1} - x_{k} } + \abs{1-\frac{c_{k+1}}{c_{k}} } \norm{x_{k} - \J_{c_{k} A}x_{k}} + \frac{ 1 - \frac{\lambda_{k}}{2}}{2}  \norm{\R_{k} x_{k} -x_{k} } +\frac{1}{2}\eta_{k} \norm{e_{k}}\\
	\stackrel{\cref{eq:lemma:RkJckINEQ:Rk}}{\leq} &   \frac{1}{2} \norm{ x_{k+1} - x_{k} } + \lr{ \abs{1-\frac{c_{k+1}}{c_{k}} }  + 1 - \frac{\lambda_{k}}{2} }\norm{x_{k} - \J_{c_{k} A}x_{k}}+\frac{1}{2}\eta_{k} \norm{e_{k}}\\
	\stackrel{\cref{eq:GPPA}}{\leq} & \frac{\lambda_{k}}{2} \norm{x_{k} - \J_{c_{k} A}x_{k}} + \lr{ \abs{1-\frac{c_{k+1}}{c_{k}} }  + 1 - \frac{\lambda_{k}}{2} }\norm{x_{k} - \J_{c_{k} A}x_{k}}+\eta_{k} \norm{e_{k}}\\
	~=~&\lr{ 1 + \abs{1-\frac{c_{k+1}}{c_{k}}}} \norm{ x_{k} - \J_{c_{k } A}x_{k } } + \eta_{k} \norm{e_{k}},
	\end{align*}
	which follows immediately by the desired inequality in  \cref{lemma:RkJckINEQ:xk+1}. 
	
	\cref{lemma:RkJckINEQ:cknonincreasing}: In view of $(\forall k \in \mathbb{N})$  $c_{k+1} \geq c_{k}$, we establish that 
	\begin{align*}
	\lr{\forall k \in \mathbb{N}} \quad \lr{ 1 + \abs{1-\frac{c_{k+1}}{c_{k}}}} =\frac{c_{k+1}}{c_{k}}.
	\end{align*}
	Hence, the required inequality is immediate from \cref{lemma:RkJckINEQ:xk+1}. 
	
	\cref{lemma:RkJckINEQ:to0}:  In view of $\sum_{k \in \mathbb{N}} \eta_{k} \norm{e_{k}} < \infty$ and $(\forall k \in \mathbb{N})$  $c_{k+1} \geq c_{k} \geq c_{0} >0$, we observe that
	\begin{align*}
	\sum_{k \in \mathbb{N}}  \frac{ \eta_{k} }{ c_{k+1} } \norm{e_{k}}  \leq \sum_{k \in \mathbb{N}}  \frac{ \eta_{k} }{ c_{0} } \norm{e_{k}}   < \infty.
	\end{align*}
	
	Bearing \cref{lemma:RkJckINEQ:cknonincreasing} in mind and applying \cref{fact:alphakINEQ}  with $(\forall k \in \mathbb{N})$ $\alpha_{k} = \frac{1}{c_{k}}  \norm{ x_{k} - \J_{c_{k } A}x_{k } }$,  $\gamma_{k} =\beta_{k} \equiv 0$, and $\varepsilon_{k} = \frac{ \eta_{k} }{ c_{k+1} } \norm{e_{k}}$, we know  that $\lim_{k \to \infty}  \frac{1}{c_{k}}  \norm{ x_{k} - \J_{c_{k } A}x_{k } }$ exists in $\mathbb{R}_{+}$. Combine this with \cref{lemma:GPPAPro}\cref{lemma:GPPAPro:lim}  to conclude that $\lim_{k \to \infty}  \frac{1}{c_{k}} \lr{x_{k} - \J_{c_{k} A}x_{k} } =0$. 	
\end{proof}

\subsection{Exact version of generalized proximal point algorithms} \label{section:exactGPPA}

In this subsection, we consider the sequence generated by \cref{eq:GPPA} with $(\forall k \in \mathbb{N})$ $e_{k} \equiv 0$ and $\eta_{k} \equiv 0$. In fact, we consider the convergence of the exact version of generalized proximal point algorithms.

Note that although in practice we normally can only get the inexact version of generalized proximal point algorithms, by invoking \cref{corollary:stabilityxkconverge},  the weak convergence (resp.\,strong convergence) of the translated exact version of  generalized proximal point algorithms can be extended to the weak convergence (resp.\,strong convergence) of their inexact versions. We shall do this extension in the following subsection.

\begin{proposition} \label{prop:exacGPPAck}
	Suppose that $(\forall k \in \mathbb{N})$ $e_{k} \equiv 0$ and $\eta_{k} \equiv 0$ in \cref{eq:GPPA}.
	Then the following assertions hold. 
	\begin{enumerate}
		\item \label{prop:exacGPPAck:INEQ} $(\forall k \in \mathbb{N})$  $\dist  \lr{0, A\lr{ \J_{c_{k } A}x_{k}}}  \leq  \norm{ \frac{1}{c_{k}} \lr{ \J_{c_{k } A}x_{k} -x_{k} } }$.

		\item \label{prop:exacGPPAck:to0}  Suppose that 	$(\forall k \in \mathbb{N})$  $c_{k+1} \geq c_{k}$ or  that $\sum_{k \in \mathbb{N}} c_{k}^{2}=\infty$ and $(\forall k \in \mathbb{N})$ $\lambda_{k} \equiv 1$. Then $(\forall k \in \mathbb{N}) $ $ \frac{1}{c_{k+1}} \norm{  \J_{c_{k+1 } A}x_{k+1} -x_{k+1} } \leq \frac{1}{c_{k}}  \norm{ \J_{c_{k } A}x_{k} -x_{k} }$.

	\item  \label{prop:exacGPPAck:1/ck} 
	Suppose that $\sum_{k \in \mathbb{N}} \lambda_{k} \lr{2 -\lambda_{k}} = \infty$ and	$(\forall k \in \mathbb{N})$  $c_{k+1} \geq c_{k}$ or  that $\sum_{k \in \mathbb{N}} c_{k}^{2}=\infty$ and $(\forall k \in \mathbb{N})$ $\lambda_{k} \equiv 1$. Then  $\frac{1}{c_{k}} \lr{ \J_{c_{k } A}x_{k} -x_{k} } \to 0$.
	\end{enumerate}
\end{proposition}

\begin{proof}

	\cref{prop:exacGPPAck:INEQ}: This follows directly from  \cref{corollary:JcA}\cref{corollary:JcA:gra}.

 	\cref{prop:exacGPPAck:to0}: If $(\forall k \in \mathbb{N})$  $c_{k+1} \geq c_{k}$, applying  \cref{lemma:RkJckINEQ}\cref{lemma:RkJckINEQ:cknonincreasing} with$(\forall k \in \mathbb{N})$ $e_{k} \equiv 0$ and $\eta_{k} \equiv 0$, we deduce $(\forall k \in \mathbb{N}) $ $ \frac{1}{c_{k+1}} \norm{  \J_{c_{k+1 } A}x_{k+1} -x_{k+1} } \leq \frac{1}{c_{k}}  \norm{ \J_{c_{k } A}x_{k} -x_{k} }$.
 	
 	If $\sum_{k \in \mathbb{N}} c_{k}^{2}=\infty$ and $(\forall k \in \mathbb{N})$ $\lambda_{k} \equiv 1$, then via \cite[Lemma~2.1]{DongRevisited2014}, we also obtain $(\forall k \in \mathbb{N}) $ $ \frac{1}{c_{k+1}} \norm{  \J_{c_{k+1 } A}x_{k+1} -x_{k+1} } \leq \frac{1}{c_{k}}  \norm{ \J_{c_{k } A}x_{k} -x_{k} }$.

\cref{prop:exacGPPAck:1/ck}: This is clear from \cref{prop:exacGPPAck:to0} above and \cref{lemma:GPPAPro}\cref{lemma:GPPAPro:lim}.  
\end{proof}

\begin{remark} \label{remark:exacGPPAck}
	Consider \cref{prop:exacGPPAck}\cref{prop:exacGPPAck:1/ck} with the assumption that $\sum_{k \in \mathbb{N}} c_{k}^{2}=\infty$ and $(\forall k \in \mathbb{N})$ $\lambda_{k} \equiv 1$. In view of  \cref{eq:GPPA}, we observe that $(\forall k \in \mathbb{N})$  $x_{k+1} = \J_{c_{k} A}x_{k} $. Taking this and \cref{prop:exacGPPAck}\cref{prop:exacGPPAck:INEQ} into account, we conclude that \cref{prop:exacGPPAck}\cref{prop:exacGPPAck:1/ck} with the assumption that $\sum_{k \in \mathbb{N}} c_{k}^{2}=\infty$ and $(\forall k \in \mathbb{N})$ $\lambda_{k} \equiv 1$  implies  \cite[Proposition~2.1]{DongComment2015}.
\end{remark}

Note that it is possible that $c_{k} \to 0$ in  \cref{prop:weaklyconvergePPA} under the assumption of (A1) although $\inf_{k \in \mathbb{N}}c_{k} >0$ is critical in
many results on the convergence of generalized proximal point algorithms in the literature.
\begin{proposition}  \label{prop:weaklyconvergePPA}
	Suppose that $(\forall k \in \mathbb{N})$ $e_{k} \equiv 0$ and $\eta_{k} \equiv 0$ in \cref{eq:GPPA} and  that one of the following holds.
	\begin{itemize}
		\item[{\rm (A1)}]  $\sum_{k \in \mathbb{N}} c_{k}^{2}=\infty$ and $(\forall k \in \mathbb{N})$ $\lambda_{k} \equiv 1$. 
		
	\item[{\rm (A2)}]  $\sum_{k \in \mathbb{N}} \lambda_{k} \lr{2 -\lambda_{k}} = \infty$, $ \sup_{k \in \mathbb{N}} c_{k} <\infty$, and $(\forall k \in \mathbb{N})$  $c_{k+1} \geq c_{k}$.
	\end{itemize}
	Then $(x_{k})_{k \in \mathbb{N}}$ converges weakly to a point in $\zer A$.
\end{proposition}

\begin{proof}
	Applying  \cref{lemma:GPPAPro}\cref{lemma:GPPAPro:exists} with $(\forall k \in \mathbb{N})$ $\eta_{k} \equiv 0$ and $e_{k} \equiv 0$, we deduce  that 
	\begin{align}\label{eq:prop:weaklyconvergeckto0:exists}
	\lr{\forall \bar{x} \in \zer A} \quad \lim_{k \to \infty} \norm{x_{k} -\bar{x}}  \text{ exists in } \mathbb{R}_{+}.
	\end{align}

	Combine  	\cref{prop:exacGPPAck}\cref{prop:exacGPPAck:1/ck}   with  our assumption  (A1) or (A2) to establish that $\frac{1}{c_{k}} \lr{x_{k} -\J_{c_{k} A}x_{k}} \to 0$,
	which, connected with \cref{lemma:ykweakcluster}, entails that 
	\begin{align} \label{eq:prop:weaklyconvergeckto0:subseteq}
  \Omega \lr{ \lr{x_{k}}_{k \in \mathbb{N}} } \subseteq \zer A.
	\end{align} 
	Altogether, recall our assumption that $\zer A \neq \varnothing$ and  invoke \cref{eq:prop:weaklyconvergeckto0:exists}, \cref{eq:prop:weaklyconvergeckto0:subseteq}, and \cite[Lemma~2.47]{BC2017} to reach the desired weak convergence of $(x_{k})_{k \in \mathbb{N}}$. 
\end{proof}

The following result illustrates a $R$-linear convergence result of the exact version of  generalized proximal point algorithms  and will be used to prove the strong convergence of the inexact version of   generalized proximal point algorithms in the next subsection. 
\begin{theorem} \label{theorem:exactGPPA}
	Suppose that $(\forall k \in \mathbb{N})$ $e_{k} \equiv 0$ and $\eta_{k} \equiv 0$ in \cref{eq:GPPA}. 	Suppose that $(\forall k \in \mathbb{N})$ $\lambda_{k} \in \left]0,2\right[\,$ and  that 
	$c:=\inf_{k \in \mathbb{N}} c_{k} >0$. 
	Then the following statements hold. 
	\begin{enumerate}
		\item  \label{theorem:exactGPPA:Fejer} $\lr{x_{k}}_{k \in \mathbb{N}}$ is Fej\'er monotone with respect to $\zer A$.
		\item  \label{theorem:exactGPPA:MSubR} Let $\bar{x} \in \zer A$.  Suppose that  $A$ is metrically subregular at $\bar{x} $ for $0 \in A\bar{x}$, i.e., 
		\begin{align} \label{eq:theorem:exactGPPA:MetricSub} 
		(\exists \kappa >0) (\exists \delta >0) (\forall x \in B[\bar{x}; \delta]) \quad \dist \lr{x, A^{-1}0} \leq \kappa \dist \lr{0, Ax}.
		\end{align}
		Suppose that  $x_{0} \in B[\bar{x};\delta]$.  Define $(\forall k \in \mathbb{N})$ $\gamma_{k}:=  1 + \frac{\kappa}{c_{k} }$ and $\rho_{k}:=\lr{  1- \frac{\lambda_{k}  \lr{2-\lambda_{k} } }{  \gamma_{k}^{2}} }^{\frac{1}{2}}$. Then the following hold. 
		\begin{enumerate}
			\item \label{theorem:exactGPPA:Distance} $\lr{\forall k \in \mathbb{N}}$ $x_{k} \in B[\bar{x};\delta]$ and $  \dist  \lr{x_{k+1}, \zer A} \leq \rho_{k} \dist  \lr{x_{k}, \zer A }$.
			\item \label{theorem:exactGPPA:moreassum} Suppose that $0< \underline{\lambda}:=\liminf_{k \to \infty} \lambda_{k} \leq \overline{\lambda}:=\limsup_{k \to \infty} \lambda_{k} <2$.
			Define $\rho:=\limsup_{k \to \infty} \rho_{k}$ and $\gamma :=\limsup_{k  \to \infty}\gamma_{k} $.
			Then the following hold. 
			\begin{enumerate}
				\item  $0\leq \rho \leq \lr{1 - \frac{\underline{\lambda} \lr{ 2-\overline{\lambda}}}{  \gamma^{2}}}^{\frac{1}{2}} <  \frac{1}{ \lr{  1 + \frac{\underline{\lambda} \lr{ 2-\overline{\lambda}}}{  \gamma^{2}} }^{\frac{1}{2}}  }  <1$.
				
				\item  There exist $K \in \mathbb{N}$ and $\mu \in \left]\rho, 1\right[$ such that   $\lr{ \forall k \geq K}$
				$  \dist  \lr{x_{k+1}, \zer A}  \leq \mu \dist  \lr{x_{k}, \zer A}$.
				\item There exist $\hat{x} \in \zer A$,  $K \in \mathbb{N}$,  and $\mu \in \left]\rho, 1\right[$  such that  
				\begin{align*}
				(\forall k \geq K) \quad \norm{x_{k} -\hat{x}} \leq 2 \mu^{k-K} \dist \lr{x_{K}, \zer A}.
				\end{align*}
				Consequently, $\lr{x_{k}}_{k \in \mathbb{N}}$ converges $R$-linearly to a point $\hat{x} \in \zer A$.
				In addition, if $\zer A$ is affine, then $\hat{x} =\Pro_{\zer A}x_{0}$.
			\end{enumerate}
			
		\end{enumerate} 
	\end{enumerate} 
	
\end{theorem}

\begin{proof}
	According to \Cref{fact:cAMaximallymonotone,fact:FixJcAzerA},  $(\forall k \in \mathbb{N})$ $\J_{c_{k} A}$ is $\frac{1}{2}$-averaged operator and 
	\begin{align*}
	(\forall k \in \mathbb{N}) \quad \lr{\Id -\J_{c_{k} A}  }^{-1}0 = \Fix \J_{c_{k} A} =\zer A.
	\end{align*}
	
	\cref{theorem:exactGPPA:Fejer}:  
	Apply \cref{theorem:KMI}\cref{theorem:KMI:Fejer} with $C=\zer A $ and $(\forall k \in \mathbb{N})$ $T_{k} =\J_{c_{k} A} $ and $\alpha_{k} =\frac{1}{2}$ to deduce 	\cref{theorem:exactGPPA:Fejer}.
	
	\cref{theorem:exactGPPA:MSubR}:  Invoking \cref{eq:theorem:exactGPPA:MetricSub} and for every $k \in \mathbb{N}$, applying \cref{lemma:metricallysubregularEQ}\cref{lemma:metricallysubregularEQ:A} with $\gamma=c_{k}$, we know that 	$(\forall k \in \mathbb{N})$ $  \Id -\J_{c_{k} A}  $ is metrically subregular at $\bar{x}$ for $0 = \lr{\Id -\J_{c_{k} A} } \bar{x}$, that is,   
	\begin{align*} 
\lr{\forall  x \in B[\bar{x}; \delta]} \quad 	\dist \lr{x, \lr{\Id -\J_{c_{k} A}}^{-1}0} \leq \lr{1 + \frac{\kappa}{c_{k} }}  \norm{ x -\J_{c_{k}  A} x}.
	\end{align*}	
	Notice  that  $ \gamma =\limsup_{k \to \infty }\gamma_{k}  \leq \lr{1 +\frac{\kappa}{c}}< \infty$. 
	Therefore, applying \cref{theorem:KMI}\cref{theorem:KMI:MSubR} with $(\forall k \in \mathbb{N})$ $T_{k}=\J_{c_{k} A}$, $\alpha_{k}=\frac{1}{2}$, $\Fix T_{k}=\zer A$, and $\delta_{k}=\delta$,  we easily obtain the desired results in \cref{theorem:exactGPPA:MSubR}. 
\end{proof}

\subsection{Weak and strong convergence of generalized proximal point algorithms} \label{section:weakstrongGPPA}

Recall that 
$A :\mathcal{H} \to 2^{\mathcal{H}}$ is maximally monotone with $\zer A \neq \varnothing$  and that 
\begin{align}\label{eq:section:weakstrongGPPAxk}
(\forall k \in \mathbb{N}) \quad x_{k+1} = \lr{1-\lambda_{k}}x_{k} +\lambda_{k} \J_{c_{k} A}x_{k} +\eta_{k}e_{k},
\end{align}
where $x_{0} \in \mathcal{H}$ and $(\forall k \in \mathbb{N})$ $\lambda_{k} \in \left[0,2\right]$, $c_{k} \in \mathbb{R}_{++}$, $\eta_{k} \in \mathbb{R}_{+}$, and $e_{k} \in \mathcal{H}$.

In this section, we consider  the weak and strong convergence of the sequence generated by conforming to the iteration scheme \cref{eq:section:weakstrongGPPAxk} for solving monotone inclusion problems.

Let $c \in \mathbb{R}_{++}$. It is clear that 
\begin{align*}
\sum_{k \in \mathbb{N}} \abs{ \frac{c_{k}}{c} -1} <\infty \Rightarrow c_{k} \to c \Rightarrow \inf_{k \in \mathbb{N}} c_{k} >0.
\end{align*}
Therefore,  assumptions in \cref{theorem:JckAweakconverckc} actually implies the popular assumption $\inf_{k \in \mathbb{N}} c_{k} >0$ for the convergence of generalized proximal point algorithms.

\begin{theorem} \label{theorem:JckAweakconverckc}
	Suppose that $(\forall k \in \mathbb{N})$  $\lambda_{k} \in \left]0,2\right]$ and $\sum_{k \in \mathbb{N}} \eta_{k} \norm{e_{k}} < \infty$. Suppose that there exists $c \in \mathbb{R}_{++}$ such that $\sum_{k \in \mathbb{N}} \abs{ \frac{c_{k}}{c} -1} <\infty$ $($e.g., $(\forall k \in \mathbb{N})$ $c_{k} \equiv c \in \mathbb{R}_{++}$$)$.	
	Then the following statements hold. 
	\begin{enumerate}
		\item \label{theorem:JckAweakconverckc:iff} 	Define   
		\begin{align*}
		(\forall i \in \mathbb{N}) \lr{\forall k \in \mathbb{N} } \quad	& \xi_{k+1}(i)=\lr{(1-\lambda_{k +i})  \Id  + \lambda_{k+i} \J_{c A} } \xi_{k}(i) \text{ and } \xi_{0}(i)=x_{i}. 
		\end{align*}
		Suppose that for every $ i \in \mathbb{N}$,
		$(\xi_{k}(i))_{k \in \mathbb{N}}$ weakly   converges to a point $\xi(i) \in \mathcal{H}$. Then the following hold. 
		\begin{enumerate}
			\item  There exists a point $\bar{\xi} \in \mathcal{H}$ such that $(\xi(i))_{i \in \mathbb{N}}$ strongly converges to $\bar{\xi}$.
			\item  $(x_{k})_{k\in \mathbb{N}}$ weakly  converges to $\bar{\xi} = \lim_{i \to \infty} \xi(i)$.  
			\item  Suppose  that $(\forall i \in \mathbb{N})$ $(\xi_{k}(i))_{k \in \mathbb{N}}$   converges strongly to a point $\xi(i) \in \mathcal{H}$. Then  $(x_{k})_{k\in \mathbb{N}}$   converges strongly to $\bar{\xi} = \lim_{i \to \infty} \xi(i)$. 
		\end{enumerate}
		
		\item  \label{theorem:JckAweakconverckc:weak} Suppose that $\sum_{k \in \mathbb{N}} \lambda_{k} \lr{2-\lambda_{k}} = \infty$.  Then $(x_{k})_{k \in \mathbb{N}} $ converges weakly to a point in $\zer A$.
	\end{enumerate}
	
\end{theorem}

\begin{proof}
	Define  
	\begin{align*}
\lr{\forall k \in \mathbb{N}} \quad 	G_{k}:=  (1-\lambda_{k})  \Id  + \lambda_{k} \J_{c A}     \text{ and }  F_{k}:= (1-\lambda_{k}) \Id + \lambda_{k}\J_{c_{k}  A}. 
	\end{align*}
	In view of \Cref{fact:cAMaximallymonotone,fact:firmlynonexpansiveaveraged,fact:FixJcAzerA}, we know that 
	$\J_{c A}  $ and  $(\forall k \in \mathbb{N})$   $\J_{c_{k}  A}$ are $\frac{1}{2}$-averaged and that 
	\begin{align} \label{eq:theorem:JckAweakconverckc:Fix}
	\Fix \J_{c A}  = \zer A \neq \varnothing.
	\end{align}
	Moreover, via \cref{fact:Averagedlambdaalpha},  we know that $(\forall k \in \mathbb{N})$    $	G_{k}$ is   nonexpansive.  
	
	Employing \cref{corollary:JcA}\cref{corollary:JcA:mulambda} in the first equality and utilizing the  nonexpansiveness of $  \J_{c_{k} A}$ in the inequality below, we observe that for every $x \in \mathcal{H}$ and every $k \in \mathbb{N}$,
%	Bearing \cref{corollary:JcA}\cref{corollary:JcA:mulambda} in mind, we observe that for every $x \in \mathcal{H}$ and every $k \in \mathbb{N}$,
	\begin{subequations} \label{eq:theorem:JckAweakconverckc:x}
		\begin{align}
		\norm{	\J_{c_{k} A}x - 	\J_{c A}x  } &=  \norm{ \J_{c_{k} A}x  - \J_{c_{k}  A} \left( \frac{c_{k}}{c}x + \left( 1- \frac{c_{k}}{c} \right)  \J_{c A}x  \right) }\\
		& \leq  \norm{x -\frac{c_{k}}{c}x - \left( 1- \frac{c_{k}}{c} \right)  \J_{c A}x }\\
	&	=  \abs{1 -  \frac{c_{k}}{c}  } \norm{x -  \J_{c A}x }.
		\end{align}
	\end{subequations}
%	where we utilize the   nonexpansiveness of $  \J_{c_{k} A}$ in the inequality above. 
	For every $k \in \mathbb{N}$, applying \cref{eq:theorem:JckAweakconverckc:x} with  $x =x_{k}$ and invoking \cref{lemma:JGammaAFix} in the first and second inequalities below, respectively, we establish that  
	\begin{align*}
\lr{\forall z \in \zer A} \lr{\forall k \in \mathbb{N}  } \quad 	\norm{ 	\J_{c_{k} A}x_{k} - 	\J_{c A}x_{k} } \leq \abs{1 -  \frac{c_{k}}{c}  }  \norm{x_{k} -  \J_{c A}x_{k} } \leq  \abs{1 -  \frac{c_{k}}{c}  }  \norm{x_{k} - z},
	\end{align*}
	which, connected with the assumptions that $\sum_{k \in \mathbb{N}} \abs{ \frac{c_{k}}{c} -1} <\infty$ and $(\forall k \in \mathbb{N})$  $\lambda_{k} \in \left]0,2\right]$, yields that 
	\begin{align}  \label{eq:theorem:JckAweakconverckc:JckJc}
	\sum_{k \in \mathbb{N}} \lambda_{k} \norm{ 	\J_{c_{k} A}x_{k} - 	\J_{c A}x_{k} } < \infty,
	\end{align} 
	since, via  \cite[Proposition~3.3(ii)]{OuyangWeakStrongGPPA2021}, we know that $(x_{k})_{k \in \mathbb{N}}$ is bounded. 
	  
	Taking \cref{eq:theorem:JckAweakconverckc:JckJc} and \cref{eq:theorem:JckAweakconverckc:Fix} into account and applying 
	\cref{theorem:TkxikConver}\cref{theorem:TkxikConver:general}  and \cref{theorem:TkxikConver}\cref{theorem:TkxikConver:Fk}, respectively, with $T=	\J_{c  A}$,    $\alpha =\frac{1}{2}$, and $(\forall k \in \mathbb{N})$ $T_{k} = 	\J_{c_{k} A}$, we obtain the required results in \cref{theorem:JckAweakconverckc:iff} and  \cref{theorem:JckAweakconverckc:weak}.
\end{proof}

In the following result, we deduce some weak convergence results on the inexact version of generalized proximal point algorithms from the corresponding results on the exact version of generalized proximal point algorithms.
\begin{theorem} \label{theorem:GPPAWeakStability}
	Suppose that  $\sum_{k \in \mathbb{N}} \eta_{k} \norm{e_{k}} < \infty$ and that  one of the following assumptions hold.
	\begin{itemize}
		\item[{\rm (A1)}]  $\sum_{k \in \mathbb{N}} c_{k}^{2}=\infty$ and $(\forall k \in \mathbb{N})$ $\lambda_{k} \equiv 1$. 
		\item[{\rm (A2)}]  $\sum_{k \in \mathbb{N}} \lambda_{k} \lr{2 -\lambda_{k}} = \infty$, $ \sup_{k \in \mathbb{N}} c_{k} <\infty$, and $(\forall k \in \mathbb{N})$  $\lambda_{k} \in \left]0,2\right]$ and  $c_{k+1} \geq c_{k}$. 
	\end{itemize}
	Then $(x_{k})_{k \in \mathbb{N}}$ converges weakly to a point in $\zer A$.
\end{theorem}

\begin{proof}
	Define 
	\begin{align*}
  \lr{\forall k \in \mathbb{N}} \quad   G_{k} := \lr{1-\lambda_{k}} \Id +\lambda_{k} \J_{c_{k} A} \quad \text{and} \quad (\forall i \in \mathbb{N})	~  \xi_{k+1}(i)=G_{k+i}  \xi_{k}(i) \text{ and } \xi_{0}(i)=x_{i}. 
	\end{align*}
	
	Then  \cref{eq:section:weakstrongGPPAxk} becomes
	\begin{align*}
(\forall k \in \mathbb{N}) \quad	x_{k+1}   = \lr{1-\lambda_{k}}x_{k} +\lambda_{k} \J_{c_{k} A}x_{k} +\eta_{k}e_{k} 
	 =G_{k} x_{k} +\eta_{k}e_{k}.
	\end{align*}
	Taking  \Cref{fact:cAMaximallymonotone,fact:firmlynonexpansiveaveraged,fact:Averagedlambdaalpha,fact:FixJcAzerA} into account, we know that $(\forall k \in \mathbb{N})$ $G_{k}$ is nonexpansive and that
	\begin{align*}
	(\forall k \in \mathbb{N}) \quad \Fix G_{k} =  \Fix \J_{c_{k} A} = \zer A \neq \varnothing,
	\end{align*}
	which, combined with  \cref{fact:MaximallyMonotoneCC}, entails that $\cap_{k \in \mathbb{N}} \Fix G_{k} =\zer A$ is nonempty closed and convex.

	In addition, notice that the sequence $\lr{G_{k}}_{k \in \mathbb{N}}$ of iteration mappings  depends only on $A$, $\lr{\lambda_{k}}_{k \in \mathbb{N}}$, and $\lr{c_{k}}_{k \in \mathbb{N}}$, and that for every $i \in \mathbb{N}$,
	\begin{align*}
	& \sum_{k \in \mathbb{N}} c_{k}^{2}=\infty \Leftrightarrow \sum_{k \in \mathbb{N}} c_{k+i}^{2}=\infty; \\
	& (\forall k \in \mathbb{N}) ~\lambda_{k} \equiv 1  \Rightarrow  (\forall k \in \mathbb{N}) ~\lambda_{k+i} \equiv 1; \\
	&\sum_{k \in \mathbb{N}} \lambda_{k} \lr{2 -\lambda_{k}} = \infty \Leftrightarrow \sum_{k \in \mathbb{N}} \lambda_{k+i} \lr{2 -\lambda_{k+i}} = \infty;  \\
	&  \sup_{k \in \mathbb{N}} c_{k} <\infty  \Leftrightarrow \sup_{k \in \mathbb{N}} c_{k+i} <\infty;   \\
	& (\forall k \in \mathbb{N})~c_{k+1} \geq c_{k}  \Rightarrow  (\forall k \in \mathbb{N})~c_{k+1+i} \geq c_{k+i}.
	\end{align*}
	Hence, for every $i \in \mathbb{N}$,  applying	\cref{prop:weaklyconvergePPA} with $(\forall k \in \mathbb{N})$ $x_{k}=\xi_{k}(i)$, $c_{k}=c_{k+i}$, and $\lambda_{k}=\lambda_{k+i}$, we know that any one of the assumptions (A1) and (A2) ensures that  $\lr{\xi_{k}(i)}_{k \in \mathbb{N}}$
	converges weakly to a point in $\zer A$.
%	Hence,  as a consequence of	\cref{prop:weaklyconvergePPA}, any one of the assumptions (A1) and (A2) ensures that for every $i \in \mathbb{N}$, the sequence $\lr{\xi_{k}(i)}_{k \in \mathbb{N}}$
%	converges weakly to a point in $\zer A$.
	
	Therefore, due to 	\cref{corollary:stabilityxkconverge}, we obtain that $(x_{k})_{k \in \mathbb{N}}$ converges weakly to a point in $\zer A$. 
\end{proof}

The following \cref{corollary:PPAinexact} demonstrates a weak convergence result on the inexact version of the proximal point algorithm. Notice that we don't require the popular assumption $\inf_{k \in \mathbb{N}} c_{k} >0$ in \cref{corollary:PPAinexact}. 
\begin{corollary} \label{corollary:PPAinexact}
Let $y_{0}$ and $(e_{k})_{k\in \mathbb{N}}$ be in $\mathcal{H}$, let $(\eta_{k})_{k \in \mathbb{N}}$ be in $\mathbb{R}_{+}$, and let $(c_{k})_{k \in \mathbb{N}}$ be in $\mathbb{R}_{++}$. Define $	\lr{\forall k \in \mathbb{N}}$ $y_{k+1} =\J_{c_{k} A}y_{k} +\eta_{k}e_{k}$.
	Suppose that $\sum_{k \in \mathbb{N}} \eta_{k} \norm{e_{k}} < \infty$ and that $\sum_{k \in \mathbb{N}} c_{k}^{2}=\infty$. Then $(y_{k})_{k \in \mathbb{N}}$ converges weakly to a point in $\zer A$.
\end{corollary}

\begin{proof}
	This is exactly \cref{theorem:GPPAWeakStability} under the assumption (A1). 
\end{proof}

To end this section, we provide some sufficient conditions for the strong convergence of generalized proximal point algorithms below. 
\begin{theorem} \label{theorem:inexactGPPAstrong}
	Let $\bar{x}$ be in $\zer A$. Suppose that $A$ is metrically subregular at $\bar{x}$ for $0 \in A\bar{x}$, i.e., 
	\begin{align*}  
	(\exists \kappa >0) (\exists \delta >0) (\forall x \in B[\bar{x}; \delta]) \quad \dist \lr{x, A^{-1}0} \leq \kappa \dist \lr{0, Ax}.
	\end{align*}
	Let $\tau \in \mathbb{R}_{++}$ such that $\tau + \sum_{k \in \mathbb{N}} \eta_{k} \norm{e_{k}} \leq \delta$ and let $x_{0} \in B[\bar{x}; \tau]$.
	Suppose that $ \inf_{k \in \mathbb{N}} c_{k} >0$, that $(\forall k \in \mathbb{N})$ $\lambda_{k} \in \left]0,2\right[\,$, and that $0 <  \liminf_{k \to \infty} \lambda_{k} \leq \limsup_{k \to \infty} \lambda_{k} < 2$.
	Then $(x_{k})_{k \in \mathbb{N}}$ converges strongly to a point in $\zer A$.
\end{theorem}

\begin{proof}
	Define 
	\begin{align*}
\lr{\forall k \in \mathbb{N}} \quad   G_{k} := \lr{1-\lambda_{k}} \Id +\lambda_{k} \J_{c_{k} A} \quad \text{and} \quad (\forall i \in \mathbb{N})	~  \xi_{k+1}(i)=G_{k+i}  \xi_{k}(i) \text{ and } \xi_{0}(i)=x_{i}. 
\end{align*}
	Bearing \Cref{fact:cAMaximallymonotone,fact:firmlynonexpansiveaveraged,fact:Averagedlambdaalpha,fact:FixJcAzerA} in mind, we deduce that $(\forall k \in \mathbb{N})$ $G_{k}$ is nonexpansive and that
	\begin{align*}
	(\forall k \in \mathbb{N}) \quad \Fix G_{k} =  \Fix \J_{c_{k} A} = \zer A \neq \varnothing,
	\end{align*}
	which, via  \cref{fact:MaximallyMonotoneCC}, guarantees that $\cap_{k \in \mathbb{N}} \Fix G_{k} =\zer A$ is nonempty closed and convex. 
	
	Moreover, applying	
		\cref{theorem:KMBasic}\cref{theorem:KMBasic:induction} with $(\forall k \in \mathbb{N})$ $T_{k}= \J_{c_{k} A} $ and $\alpha_{k} =\frac{1}{2}$,  	
		we observe that 
	\begin{align*}
	(\forall i \in \mathbb{N}) \quad  \norm{ \xi_{0}(i) -\bar{x}} =\norm{ x_{i} -\bar{x}}  \leq  \norm{x_{0} -\bar{x}} + \sum_{k \in \mathbb{N}} \eta_{k} \norm{e_{k}} \leq \delta. 
	\end{align*}

	Hence,	for every $i \in \mathbb{N}$, invoking  our assumptions and applying \cref{theorem:exactGPPA}\cref{theorem:exactGPPA:moreassum}\textcolor{blue}{iii.\,}with $(\forall k \in \mathbb{N})$ $\lambda_{k}=\lambda_{k+i}$, $c_{k}=c_{k+i}$, and $x_{k} =\xi_{k}(i)$,   we deduce that 
	$\lr{\xi_{k}(i) }_{k \in \mathbb{N}}$ converges $R$-linearly (hence, converges strongly) to a point in $\zer A$.

	Furthermore, due to 	\cref{corollary:stabilityxkconverge}, we deduce that $(x_{k})_{k \in \mathbb{N}}$ converges strongly to a point in $\zer A$.
\end{proof}

\section{Conclusion and Future Work} \label{section:Conclusion}

In this work, applying  a specific instance of  \cite[Lemma~2.1]{Lemaire1996} which originates from  \cite[Remark~14]{BrezisLions1978}, we  deduce  the weak convergence (resp.\,strong convergence) of the iteration method (involving nonexpansive operators) with both approximation and perturbation from the weak convergence (resp.\,strong convergence) of associated translated basic methods.  Based on this result and some beautiful properties of averaged operators, by employing some convergence results on the classical Krasnosel'ski\v{\i}-Mann iterations, we establish the weak and strong convergence of relaxation variants of Krasnosel'ski\v{\i}-Mann iterations for finding a fixed point of the associated nonexpansive operators. In addition, bearing the powerful properties of the resolvent of maximally monotone operators in mind, we apply all previous convergence results and proof techniques to generalized proximal point algorithms and obtain some weak and strong convergence results on generalized proximal point algorithms for solving monotone inclusion problems.

Note that relaxation variants of Krasnosel'ski\v{\i}-Mann iterations  and generalized proximal point algorithms cover generalized versions of multiple popular optimization algorithms. 
In the future, by specifying the nonexpansive operators in relaxation variants of Krasnosel'ski\v{\i}-Mann iterations or the maximally monotone operators in generalized proximal point algorithms accordingly, 
we shall elaborate applications of our convergence results for more particular algorithms such as the  Douglas-Rachford splitting algorithm, the three-operator splitting schemes,    the alternating direction method of multipliers, and so on. 

%\section*{Acknowledgements}

\addcontentsline{toc}{section}{References}

\bibliographystyle{abbrv}

\end{document}